\documentclass[10pt]{amsart}
\usepackage{latexsym}
\usepackage{amsmath}
\usepackage{amssymb}
\usepackage{mathrsfs}
\usepackage{graphicx}
\usepackage{color}
\usepackage{pgfpages}
\usepackage{ifthen}
\usepackage{leftidx,tensor}
\usepackage[T1]{fontenc}
\usepackage[latin1]{inputenc}
\usepackage{mathtools}
\usepackage{comment}
\usepackage{dsfont}
\usepackage[nocompress]{cite}
\usepackage[shortlabels]{enumitem}
\usepackage{aliascnt}
\usepackage[bookmarks=true,pdfstartview=FitH, pdfborder={0 0 0}, colorlinks=true,citecolor=red, linkcolor=blue]{hyperref}
\usepackage{bbm}
\usepackage{nicefrac}

\theoremstyle{plain}
\newtheorem{thm}{Theorem}[section]
\newaliascnt{cor}{thm}
\newaliascnt{prop}{thm}
\newaliascnt{lem}{thm}
\newtheorem{cor}[cor]{Corollary}

\newtheorem{lem}[lem]{Lemma}
\aliascntresetthe{cor}
\aliascntresetthe{prop}
\aliascntresetthe{lem} 
\theoremstyle{definition}
\newaliascnt{defn}{thm}
\newaliascnt{asu}{thm}
\newaliascnt{con}{thm}

\newtheorem{asu}[asu]{Assumption}

\aliascntresetthe{defn}
\aliascntresetthe{asu}
\aliascntresetthe{con}
\newcounter{stp}
\newcounter{stpi}
\newcounter{stpci}
\newcounter{stpiii}

\newtheorem{step}[stp]{Step}

\theoremstyle{thm}
\newaliascnt{rem}{thm}
\newaliascnt{exa}{thm}
\newaliascnt{masu}{thm}
\newaliascnt{nota}{thm}
\newaliascnt{sett}{thm}
\newtheorem{rem}[rem]{Remark}

\aliascntresetthe{rem}
\aliascntresetthe{exa}
\aliascntresetthe{masu}
\aliascntresetthe{nota}
\aliascntresetthe{sett}
\numberwithin{equation}{section}

\setlist[enumerate]{font = \normalfont}

\newcommand {\R}	{\mathbb{R}}

\newcommand {\E}	{\mathbb{E}}

\newcommand {\T}	{\mathbb{T}}

\renewcommand{\d}{\, \mathrm{d}}

\DeclareMathOperator{\divH}{div_{\H}}

\newcommand{\Hinfty}{\mathcal{H}^\infty}

\newcommand{\D}{\mathrm{D}}
\newcommand{\rN}{\mathrm{N}}
\renewcommand{\H}{\mathrm{H}}
	\newcommand{\dk}[1]{\partial_{#1}}
	\newcommand{\dt}{\dk{t}} 
	\newcommand{\dz}{\dk{z}} 
	
	\newcommand{\eps}{\varepsilon}
	\renewcommand{\phi}{\varphi}
	
	\renewcommand{\bar}[1]{\overline{#1}}

	\renewcommand{\div}{\mathrm{div} \, }
	\newcommand{\nablaH}{\nabla_{\H}}

    \newcommand{\Cof}{\mathrm{Cof}}
	
	\newcommand{\rC}{\mathrm{C}}
	\newcommand{\rL}{\mathrm{L}}
	\newcommand{\rW}{\mathrm{W}}
	\newcommand{\rH}{\H}
   
	\newcommand{\rB}{\mathrm{B}}

	\newcommand{\varrhobar}{\bar{\varrho}}

	\newcommand{\rX}{\mathrm X}

\makeatletter
\renewcommand{\footnoterule}{} % removes the line above first-page footnotes (subjclass/keywords/etc.)
\makeatother

\title[Heat-Conducting, Compressible primitive equations]{Global Strong Well-posedness of the heat-conducting, compressible primitive Equations}

\author{Tarek Z\"{o}chling}
\address{Technische Universit\"{a}t Darmstadt,
Schlo\ss{}gartenstra{\ss}e 7, 64289 Darmstadt, Germany.}
\email{zoechling@mathematik.tu-darmstadt.de}

\begin{document}
\keywords{Heat-conducting, compressible primitive equations, global existence and uniqueness of strong solutions close to Equilibria. \\
Tarek Z\"ochling gratefully acknowledge the support by the Deutsche Forschungsgemeinschaft (DFG) through the Research Unit FOR~5528
}

\maketitle
% =========================================================
% Setting and notation
\begin{abstract}
    The full heat-conducting compressible primitive equations are considered, extending the compressible primitive-equation framework by coupling the temperature through the ideal gas law and the thermal energy balance in the presence of gravity. Global strong well-posedness is established for small perturbations of an equilibrium state, thereby providing a result beyond the isothermal regime. The proof relies on a structural representation of the density in terms of the temperature and on a new Lagrangian transformation.
\end{abstract}
\makeatletter
\renewcommand{\footnoterule}{%
  \kern 2pt
  \hrule width \textwidth
  \kern 4pt
}
\makeatother

\begingroup
\makeatletter
\renewcommand{\thefootnote}{}% no marker text
\long\def\@makefntext#1{\noindent #1}% removes the usual footnote indent/label box
\makeatother
\footnotetext{%
\footnotesize
\noindent\rule{\linewidth}{0.4pt}\par
\smallskip
After finishing this manuscript, I become aware of the very recent article \cite{KLLT:26} by  R.~Klein, J.~Li, X.~Liu, and E.~S.~Titi (arXiv:2602.19801) dealing with a closely related subject.}
\addtocounter{footnote}{-1}
\endgroup
\section{Introduction}
\label{subsec:full-cpe-temperature}
\noindent
The primitive equations are among the fundamental models of geophysical fluid dynamics and play a central role in modern meteorology and weather prediction. They are derived from the Navier--Stokes equations in the hydrostatic regime and provide the standard large-scale description of atmospheric and oceanic flows. A defining feature of the primitive equations, in contrast to the Navier--Stokes system, is that the vertical momentum equation is not treated as an evolution equation for the vertical velocity, instead, it is replaced by the \emph{hydrostatic balance}
\[
\partial_z p = -\varrho g,
\]
which expresses that, in the vertical direction, the pressure gradient is balanced by gravity. Here \(p\) denotes the pressure, \(\varrho\) the density, and \(g\) the gravitational constant. A systematic mathematical analysis of these equations was initiated in the seminal works of Lions, Temam, and Wang~\cite{LTW:92a,LTW:92b,LTW:95}. 

\medskip

\noindent
In this work, the \emph{compressible primitive equations} with gravity are studied, and \emph{global strong well-posedness} is established for initial data given by small perturbations of an equilibrium state. Indeed, density, pressure, and temperature evolve in a coupled way through the ideal gas law and the thermal energy balance, and the resulting system carries the correct energetic structure.
Let $\Omega=\T^2\times(0,1)$ and let $\Omega_\tau:=\Omega\times(0,\tau)$ denote the associated parabolic cylinder for a fixed time horizon $0<\tau<\infty$. Consider the compressible primitive equations on $\Omega_\tau$, given by
\begin{equation}
	\left\{
	\begin{aligned}
        \dt \varrho +\divH  (\varrho v) + \dz (\varrho w)  &=0 &&\text{ in }\Omega_\tau, \\
		\varrho\, \bigl (\dt  v + u \cdot \nabla v \bigr )  - \mu \, \Delta v- \mu'\, \nablaH \divH v +  \nablaH p &=0 &&\text{ in } \Omega_\tau,  \\
		  \dz p  &= -g \varrho   &&\text{ in } \Omega_\tau, \\
        p &= R\, \varrho\, \Theta &&\text{ in } \Omega_\tau,  \\
c_\nu \varrho \, \big(\partial_t \Theta+ u \cdot \nabla \Theta\big)
+p\; \div u -\kappa \, \Delta \Theta &= Q+ \Phi &&\text{ in } \Omega_\tau,  \\
\varrho(0)=\varrho_0, \ v(0)=v_0,   \ \Theta(0) &= \Theta_0.
	\end{aligned}
	\right.
	\label{eq:full prim}\tag{\textcolor{blue}{CPE}}
\end{equation}
Here the unknowns are the density $\varrho \colon \Omega_\tau \to \R$, the velocity field $u=(v,w)\colon \Omega_\tau \to \R^3$, the pressure $p\colon \Omega_\tau \to \R$, and the temperature $\Theta\colon \Omega_\tau \to \R$. Moreover, $R>0$ denotes the (specific) ideal gas constant, $c_\nu>0$ the specific heat at constant volume, and $\kappa>0$ the heat conductivity,  while \(Q\) and \(\Phi\) denote external diabatic heating/cooling and viscous dissipation, respectively.  Here we use $\nablaH$, $\divH$ for the horizontal gradient and the horizontal divergence, that is $\nablaH:=(\partial_x,\partial_y)^\top$ and
$\divH:=\nablaH \cdot{}$. In the following, we assume that the viscosity coefficients satisfy $\mu>0$ and $\mu+\mu'>0$. The system is complemented by the boundary conditions
\begin{equation}\label{eq: bc}
\partial_z v|_{z=0,1}=0, \quad w|_{z=0,1}=0\ \text{ and }\ \partial_z \Theta|_{z=0,1}=0.
\end{equation}
An important point for the present work is that \eqref{eq:full prim} is not only a formally natural extension of the isothermal compressible primitive equations, but also the energetically correct model once the thermodynamic equation is closed in accordance with the first law. More precisely, for smooth solutions and $Q\equiv 0$, we choose the viscous heating $\Phi$ to coincide with the mechanical dissipation generated by the viscous terms in the momentum equation, namely
\begin{equation}\label{eq:phi-choice}
\Phi \;:=\; \mu\,|\nabla v|^2 \;+\; \mu'\,|\divH v|^2 .
\end{equation}
With this closure, the total energy, kinetic, internal, and gravitational potential energy, is conserved
\begin{equation}\label{eq:total-energy-def}
\mathcal E(t)
:=\int_\Omega\Big(\frac12\,\varrho|v|^2 \;+\; c_\nu \varrho\,\Theta \;+\; g\varrho z\Big)\d x
\qquad
\partial_t \mathcal E=0.
\end{equation}
Formally, this identity is obtained by testing the horizontal momentum equation with $v$, integrating the temperature equation over $\Omega$, and combining the result with the continuity equation weighted by $gz$. 
This shows that \eqref{eq:full prim} is the appropriate heat-conducting atmospheric model from the energetic point of view and provides the correct starting point for a rigorous well-posedness theory beyond the isothermal regime. A key consequence of combining the ideal gas law with the hydrostatic balance is that the pressure exhibits an explicit vertical dependence through the temperature, and this dependence is both \emph{nonlinear} and \emph{nonlocal}. More precisely, one has
\begin{equation*}
    p(t,x,y,z) = p_s(t,x,y)\,\exp\!\Big(-\int_0^z \frac{1}{\Theta(\cdot,\eta)}\,\d\eta\Big),
\end{equation*}
where \(p_s\) denotes the pressure at the surface, see also \autoref{sec : prelim} for more details. We note that without gravity, i.e. \(g=0\), this relation reduces to \(p=p_s\), which makes the analysis substantially simpler. In particular, in this case one can average the temperature equation without generating nonlocal temperature contributions, which leads to a substantially simpler representation of the vertical velocity.
Thus, incorporating gravity is not only physically natural, but also introduces a genuinely additional mathematical difficulty.

\medskip

\noindent
Let us now comment on previous work and categorize our result. For the \emph{incompressible} primitive equations Lions, Temam and Wang established the existence of global weak solutions, constructed via Galerkin approximation, see \cite{LTW:92a,LTW:92b,LTW:95}. 
% In the \emph{compressible} setting, they work in \emph{pressure coordinates}, which essentially transform the system into one of incompressible type and thereby allow them to prove existence of weak solutions by adapting the analysis from the incompressible primitive equations~\cite{LTW:92a,LTW:95}. We emphasize, however, that this pressure coordinate approach is not mathematically rigorous, and uniqueness of weak solutions for the incompressible primitive equations remains open.
The well-posedness theory was later substantially advanced, most notably by Cao and Titi~\cite{CT:07}, who established global well-posedness in three dimensions for arbitrarily large \( \rH^1 \) data. In \cite{HK:16,GGHHK:20,GGHHK:20b}, the theory was further extended to initial data in the critical spaces \(\rB^{2/q}_{qp}\), as well as to continuous initial data subject to small \(\rL^\infty\) perturbations.

\medskip

\noindent
By contrast, the compressible case remains much less understood. Under Neumann boundary conditions, local and global strong well-posedness for small data in the absence of gravity ($g=0$) were established in \cite{LT:20,LT:21}. These results were later extended in \cite{HIRZ:25} to include gravity together with the isothermal ideal gas law. More recently, in \cite{HJMZ:25}, we proved global well-posedness of the two-dimensional compressible primitive equations for arbitrarily large data in \(\rH^1\). Unlike in the incompressible setting, however, all currently available well-posedness results for the compressible primitive equations are restricted to the \emph{isothermal} regime.

\medskip

\noindent
Concerning the heat conducting compressible Navier-Stokes equations, for weak-solution theory we refer to  Bresch, Desjardins \cite{BD:07} or Feireisl, Novotn\'y  \cite{FN:17}. For strong solutions a classical reference is Matsumura, Nishida  \cite{MN:83}.
\subsection{Analytical strategy}

Known approaches to the compressible primitive equations exploit in an essential way that the continuity equation admits a \textit{dimension reduction}. In the isothermal setting, the hydrostatic balance $\partial_z p=-\varrho g$ together with a polytropic pressure law $p=\varrho^\gamma$ implies an explicit vertical profile for the density. In regimes without gravity this simplification becomes even more pronounced, since the vertical dependence is then trivial; see, for instance, \cite{LT:20,LT:21, HIRZ:25}. 
In the non-isothermal case, however, the temperature prevents such a direct reduction. Indeed, the ideal-gas law $p=R\varrho\Theta$ introduces an additional vertical dependence through $\Theta$, so that neither the density nor the pressure can be expressed in a simple separated form. As a consequence, recovering the vertical velocity $w$ from the continuity equation becomes substantially more delicate.

\medskip

\noindent
Our approach relies on a careful exploitation of the intrinsic structure of the system. A key observation is that the combination of the ideal gas law and the hydrostatic balance allows one to express the vertical dependence of the density explicitly in terms of the temperature through a nonlinear functional. More precisely, one can write
\begin{equation*}
    \varrho(t,x,y,z) = \bar{\varrho}(t,x,y)\,\hat{B}(\Theta(t,x,y,z)),
\end{equation*}
where
\begin{equation*}
    \hat{B}(\Theta(t,x,y,z))
    := \frac{B(\Theta(t,x,y,z))}{\bar{B}(\Theta(t,x,y))}
    = \frac{1}{\Theta(t,x,y,z)}
    \,\frac{\exp\!\Big(-\int_0^z \frac{1}{\Theta(\cdot,\eta)}\,\d\eta\Big)}
    {1-\exp\!\Big(-\int_0^1 \frac{1}{\Theta(\cdot,\eta)}\,\d\eta\Big)}.
\end{equation*}
Here, \(\bar{\varrho}\) denotes the vertical average of the density, i.e. $\bar{\varrho}=\int_0^1 \varrho(\cdot,\eta)\,\d\eta.$
This representation shows that the averaged density \(\bar{\varrho}\) satisfies the evolution equation
\begin{equation*}
    \dt \bar{\varrho}
    + \divH\!\Big(
        \bar{\varrho}\,\int_0^1 (\hat{B}(\Theta)\,v)(\cdot,\eta)\,\d\eta
      \Big)
    =0,
\end{equation*}
see also \autoref{sec : prelim}. In particular, the full continuity equation is no longer used as an independent evolution equation for the full density, but instead yields an evolution equation for the averaged density together with a representation formula for the vertical velocity
\begin{equation*}
   ( \varrho\, w)(t,x,y,z)
    = -\int_0^z
    \big[\dt \varrho + \divH(\varrho\,v)\big](t,x,y,\eta)\,\d\eta.
\end{equation*}
A major difficulty, compared to the isothermal regime (cf.~\cite{HIRZ:25,LT:20}), is that the time-derivative contributions do \emph{not} cancel. Instead, they generate time derivatives of the nonlinear functional \(\hat{B}\), which must be incorporated into the full linearization of the system. As a consequence, one inevitably obtains \emph{nonlocal} time-derivative contributions in the temperature equation.
To handle the above density structure encoded by \(\hat{B}\), we introduce a completely new Lagrangian transformation associated with the horizontal flow field
\[
b(t,x,y) :=\int_0^1 (\hat{B}(\Theta)\,v)(t,x,y,\eta)\,\d\eta.
\]
We then show that this transformation is well defined in regimes where the temperature remains nondegenerate. In our setting, this is ensured by choosing the initial temperature sufficiently close to an equilibrium state.

\medskip

\noindent
Once this Lagrangian transformation is available, we analyze the resulting fully linearized system. This requires a careful linearization of the nonlinear functional \(\hat{B}\) and its Fr\'echet derivative \(D\hat{B}\). In addition, the representation formula for \(w\) has to be decomposed into linear contributions, which enter the linearized system, and genuinely nonlinear remainder terms.
This leads to a linear system containing a nonlocal contribution in the time derivative, explicitly of the form
\begin{equation*}
    \mathcal{L}(\Theta)
    := 2\dt \Theta - \beta(z)\int_0^1 \dt \Theta(\cdot,\eta)\,\d\eta,
\end{equation*}
for a suitable nondegenerate function \(\beta=\beta(z)\) depending on the equilibrium temperature profile and the vertical variable. Exploiting the special structure of \(\mathcal{L}\), we prove that \(\mathcal{L}\) is invertible and are thus led to consider the abstract operator
\begin{equation*}
   A := -\mathcal{L}^{-1}(\alpha \Delta),
\end{equation*}
where \(\alpha=\alpha(z)\) is a variable nondegenerate coefficient. The operator \(A\) is then treated as a noncommuting product of operators. Sectoriality is established by means of \cite{Weber:98}, after verifying a suitable commutator estimate. Finally, once the linear theory is in place, the nonlinear terms are estimated in the appropriate function spaces, and the solution is constructed by a contraction mapping argument.

% \subsection{Outlook towards thermodynamic and moist coupling}

%  Once a temperature-consistent atmospheric theory is in place, it opens the door to revisit the \textit{fully coupled} atmosphere-ocean system with thermodynamic coupling at the air-sea interface (e.g.\ consistent heat fluxes and energy budgets across $\Gamma_i$), and to incorporate temperature also on the ocean side, see \cite{LTW:93, LTW:95, BBH:25, Z:25}. In parallel, one may consider an atmosphere-only extension of the compressible primitive equations in which the temperature equation is coupled with \emph{moisture dynamics} (humidity transport, phase transitions, and latent heat effects). For this moist extension as well, the currently available mathematical results rely on a reformulation in {pressure coordinates}, and constant pressure at the boundaries, see also \cite{HKLT:23}. For related results concerning moisture dynamics coupled to the compressible Navier-Stokes equations see \cite{BHMZ:26a, DKLT:25}. Both the full coupled atmosphere-ocean system and the compressible primitive equations coupled to moisture dynamics will be considered in future studies. 

 \medskip

 \noindent
 The article is structured as follows. In \autoref{sec : prelim}, we collect the preliminaries and reformulate the system using the explicit structural representation described above. In \autoref{sec: main}, we state the main local and global well-posedness results. In \autoref{sec: Lagrange}, we introduce the Lagrangian transformation, reformulate the system accordingly, and prove that the transformation is well defined. Next, in \autoref{sec: local}, we establish the well-posedness of the linearized system and derive estimates for the nonlinear terms. Finally, in \autoref{sec: proof} we conclude global existence by a contraction mapping argument.

\section{Preliminaries}\label{sec : prelim}
\noindent
In the following we assume, without loss of generality, that all constants are normalized to one and that $Q=\Phi=0$ and that the viscosity coefficients $\mu$, $\mu'$ satisfy the standard condition $\mu>0$ and $\mu + \mu'>0$. Note that in particular we can also choose $\Phi= \Phi (v) = \mu\,|\nabla v|^2 \;+\; \mu'\,|\divH v|^2$, see \autoref{rem: main}. 
\noindent
Combining hydrostatic balance with the equation of state yields an explicit representation of the pressure
\begin{equation}
    \label{eq: pressure}
    p(t,x,y,z) = p_s(t,x,y) \, \mathrm{exp} \big ( {-\int_0^z \frac{1}{\Theta(\cdot, \eta)} \d \eta } \big ), \ \text{ where } \ p_s(t,x,y) := p(t,x,y,z=0).
\end{equation}
In particular, the pressure at the top boundary, $p_t(t,x,y) \coloneqq p(t,x,y,z=1)$, is completely determined by $p_s$ and $\Theta$. More precisely,
\begin{equation*}
    p_t(t,x,y) = p_s(t,x,y)\, \mathrm{exp} \big ( {-\int_0^1 \frac{1}{\Theta(\cdot, \eta)} \d \eta } \big ).
\end{equation*}
Using the ideal gas law with $R=1$, the density can be written as
\begin{equation}
    \label{eq: density}
    \varrho(t,x,y,z) = p_s(t,x,y)\, B(\Theta(t,x,y,z)), 
\end{equation}
where the functional $B$ is defined by
\begin{equation}
    \label{eq func B}
    B(\Theta(t,x,y,z)) := \frac{1}{\Theta(t,x,y,z)} \, \mathrm{exp} \big ( {-\int_0^z \frac{1}{\Theta(\cdot, \eta)} \d \eta } \big ).
\end{equation}
Denoting by $\bar{f}$ the vertical average of an integrable function $f$, that is  $\bar{f}=\int_0^1 f(\cdot,\eta)\,\d\eta$, we infer
\begin{equation}\label{eq: avg density}
    \bar{\varrho}(t,x,y) = p_s(t,x,y) \, \bar{B}(\Theta(t,x,y)) = p_s(t,x,y) \Big ( 1- \mathrm{exp} \big ( {-\int_0^1 \frac{1}{\Theta(\cdot, \eta)} \d \eta } \big ) \Big ),
\end{equation}
where we used the identity
\begin{equation*}
    \frac{\d}{\d z} \Big (\mathrm{exp} \big ( {-\int_0^z \frac{1}{\Theta(\cdot, \eta)} \d \eta } \big ) \Big ) = -B(\Theta(t,x,y,z)).
\end{equation*}
In particular, this coincides with the representation obtained by vertically averaging hydrostatic balance, namely $\bar{\varrho}(t,x,y)=p_s(t,x,y)-p_t(t,x,y)$.
Next, we introduce the normalized functional $\hat{B}$, defined as the ratio between $B$ and its vertical mean:
\begin{equation}
    \label{eq: hat B}
    \hat{B}(\Theta(t,x,y,z)) := \frac{B(\Theta(t,x,y,z))}{\bar{B}(\Theta(t,x,y))} = \frac{1}{\Theta(t,x,y,z)} \, \frac{\mathrm{exp} \big ( {-\int_0^z \frac{1}{\Theta(\cdot, \eta)} \d \eta } \big )}{1- \mathrm{exp} \big ( {-\int_0^1 \frac{1}{\Theta(\cdot, \eta)} \d \eta } \big )} \ \text{ and } \ \int_0^1 \hat{B}(\Theta(\cdot,\eta)) \d \eta =1.
\end{equation}
With this notation, the density can equivalently be expressed as
\begin{equation*}
    \varrho(t,x,y,z) = \bar{\varrho}(t,x,y)\, \hat{B}(\Theta(t,x,y,z)).
\end{equation*}

\noindent
Motivated by this identity, we vertically average the continuity equation and obtain
\begin{equation}
    \label{eq: avg cont}
    \dt \varrhobar + \divH (\varrhobar \, b) =0, \ \text{ where } \ b(t,x,y) := \int_0^1 (\hat{B}(\Theta) \, v ) (\cdot, \eta) \d \eta.
\end{equation}
In the hydrostatic setting, the un-averaged continuity equation is no longer employed as an independent evolution law. Instead, once $\varrhobar$ and $\Theta$ are represented as above, it provides a diagnostic formula for the vertical velocity $w$ in terms of $\varrhobar$, $\Theta$, and $v$. More precisely, using the density representation, we infer that $w$ is given by
\begin{equation}\label{eq:W-noQt}
\begin{aligned}
\big (\varrhobar\, \hat{B}(\Theta) \,w\big )(t,x,y,z)
=-\int_{0}^{z}\Big[
\hat B(\Theta)\,(v-b)\cdot\nablaH \varrhobar
+\varrhobar\, \big(\partial_t\hat B(\Theta)+\divH(\hat B(\Theta)\, v)-\hat B(\Theta)\,\divH b\big)
\Big](\cdot,\eta) \d\eta.
\end{aligned}
\end{equation}
We emphasize that the boundary conditions $w|_{z=0,1}=0$ are automatically satisfied by this representation whenever $\varrhobar$ solves the averaged continuity equation \eqref{eq: avg cont}.

With these preparations, the system \eqref{eq:full prim} can be recast in the following form
\begin{equation}
	\left\{
	\begin{aligned}
        \dt \varrhobar +\divH  (\varrhobar \, b)  &=0 &&\text{ in }\T^2_\tau, \\
		\varrhobar \, \hat{B}(\Theta) \bigl (\dt  v + u \cdot \nabla v \bigr )  - \mu \Delta v- \mu' \nablaH \divH v +  \nablaH  ( \varrhobar\, \hat{B}(\Theta) \, \Theta) &=0 &&\text{ in } \Omega_\tau,  \\
 \varrhobar\, \hat{B}(\Theta) \big(\partial_t \Theta+ u \cdot \nabla \Theta\big)
+\varrhobar\, \hat{B}(\Theta) \, \Theta\, \div u - \Delta \Theta &=0 &&\text{ in } \Omega_\tau,  \\
\varrhobar(0)=\varrhobar_0 := \int_0^1 \varrho_0(\cdot,\eta) \d \eta, \ v(0)=v_0,   \ \Theta(0) &= \Theta_0.
	\end{aligned}
	\right. 
	\label{eq:full prim reformulated}
\end{equation}
In this formulation, the prognostic variables are the averaged density $\varrhobar$, the horizontal velocity $v$, and the temperature $\Theta$. The remaining quantities are determined diagnostically: the pressure is given by $p=\varrhobar\,\hat{B}(\Theta)\,\Theta$ with $\hat{B}(\Theta)$ as in \eqref{eq: hat B}, the averaged flux is $b=\int_0^1(\hat{B}(\Theta)\,v)(\cdot,\eta)\,\d\eta$, and the vertical velocity $w$ is recovered from \eqref{eq:W-noQt}.

\section{Main Results}\label{sec: main}
\noindent
We are now in the position to formulate our main theorem concerning the global, strong well-posedness for initial data close to a steady state of the system \eqref{eq:full prim} subject to the boundary conditions \eqref{eq: bc}. Our requirements on the set of initial data are collected in the following assumption.

\begin{asu}\label{assu:data0}
Let $\tau>0$ and $\varrhobar^\ast, \Theta^\ast>0$. Assume that 
\begin{equation*}
    (\varrhobar_0, v_0, \Theta_0) \in \rH^3(\T^2) \times \rH^3(\Omega)^2 \times \rH^3(\Omega)
\end{equation*}
satisfying $\dz v_0 |_{z=0,1}= \dz \Theta_0 |_{z=0,1}=0$. Moreover, assume that 
\begin{equation*}
    \| (\varrhobar-\varrhobar^\ast, v_0, \Theta_0 - \Theta^\ast) \|_{\rH^3(\T^2) \times \rH^3(\Omega)^2 \times \rH^3(\Omega)} \leq \eps.
\end{equation*}
\end{asu}
\noindent
The main theorem of this article then reads as follows.
\begin{thm}[Global, strong well-posedness of \eqref{eq:full prim}]\label{thm main} \mbox{}\\
Let $\tau>0$ and $\varrhobar^\ast,\Theta^\ast>0$. Assume that the initial data
$(\varrhobar_0,v_0,\Theta_0)$ satisfy \autoref{assu:data0}. Then there exists
\[
\eps_0=\eps_0(\tau,\varrhobar^\ast,\Theta^\ast,\Omega)>0
\]
such that for every $\eps\in(0,\eps_0)$ the reformulated system
\eqref{eq:full prim reformulated} subject to the boundary conditions
\eqref{eq: bc} admits a unique strong solution $(\varrhobar,v,\Theta)$ with
\begin{equation*}
\begin{aligned}
\varrhobar &\in \rH^{1}(0,\tau;\rH^2(\T^2)) \cap \rL^\infty(0,\tau;\rH^{3}(\T^2)), \\
v &\in \rH^{1}(0,\tau;\rH^2(\Omega;\R^2)) \cap \rL^2(0,\tau;\rH^{4}(\Omega;\R^2)), \\
\Theta &\in \rH^{1}(0,\tau;\rH^2(\Omega)) \cap \rL^2(0,\tau;\rH^{4}(\Omega)).
\end{aligned}
\end{equation*}
With $\hat{B}(\Theta)$ as in \eqref{eq: hat B}, define the diagnostic quantities
\[
\varrho := \varrhobar\,\hat{B}(\Theta), 
\quad 
p := \varrho\,\Theta
\quad \text{ and } \quad
w \text{ by } \eqref{eq:W-noQt}.
\]
Then $(\varrho,u=(v,w),p,\Theta)$ is a unique, strong solution of the full compressible
primitive equations \eqref{eq:full prim}. Moreover, for all $(t,x,y,z)\in[0,\tau]\times\Omega$,
\begin{equation*}
\frac{\varrhobar^\ast \hat{B}(\Theta^\ast)(z)}{2} \le \varrho(t,x,y,z)
\le \frac{3\,\varrhobar^\ast \hat{B}(\Theta^\ast)(z)}{2} \quad \text{ and } \quad
\frac{\Theta^\ast}{2} \le \Theta(t,x,y,z) \le \frac{3\,\Theta^\ast}{2}.
\end{equation*}
\end{thm}
\begin{rem}\label{rem: main}
    { \rm The statement of \autoref{thm main} is presented for the simplified case $Q=\Phi=0$. We stress, however, that the same conclusion remains valid if one allows for a prescribed diabatic heating term $Q=Q(t,x,y,z)\in \rL^2(0,\tau;\rH^2(\Omega))$ and includes viscous dissipation, with $\Phi$ taken in the standard form
\begin{equation*}
    \Phi(v) = \mu \, | \nablaH v |^2 + \mu' \, | \divH v |^2.
\end{equation*}}
\end{rem}

\section{Coordinate Transformation}\label{sec: Lagrange}
\noindent
In view of the transport structure of the averaged continuity equation \eqref{eq: avg cont}, it is natural to introduce characteristics associated with the flow field $b$, given by
\begin{equation}
    \label{eq: b}
     b= \int_0^1 (\hat{B}(\Theta) \, v ) (\cdot, \eta) \d \eta \ \text{ with } \  \hat{B}(\Theta) =  \frac{1}{\Theta} \, \frac{\mathrm{exp} \big ( {-\int_0^z \frac{1}{\Theta(\cdot, \eta)} \d \eta } \big )}{1- \mathrm{exp} \big ( {-\int_0^1 \frac{1}{\Theta(\cdot, \eta)} \d \eta } \big )} .
\end{equation}
Define the flow $\rX$ as the solution of the differential equation 
\begin{equation}
\label{eq:flow}
\left \{
    \begin{aligned}
        %\begin{cases}
            \dt \mathrm{X}(t,x,y) &= b (t,\mathrm{X}(t,x,y)), \ t>0 ,\\
            \mathrm{X}(0,x,y) &= (x,y),
       % \end{cases}
    \end{aligned}
    \right. 
\end{equation}
for $(x,y) \in \T^2$. Note that $\mathrm{X}$ is only a
two dimensional flow. Solving \eqref{eq:flow} leads to
\begin{equation}\label{eq:diffeo X Euler Lagrange sea ice para-hyper}
    \rX(t,x,y) = (x,y)^\top + \int_0^t \ b(s,\rX(s,x,y)) \d s 
\end{equation}
for $(x,y) \in \T^2$. In the following, we show that the flow $\mathrm{X}(t,\cdot) : \T^2 \rightarrow \T^2$ is a well-defined $\rC^1$-diffeomorphisms provided that the temperature $\Theta$, the vertical velocity $v$ and hence $b$ are regular enough. To be precise, let $\tau >0$ and assume that $(v, \Theta) \in\E^v_1(0,\tau) \times \E^\Theta_1(0,\tau)$ where
\begin{equation*}\label{eq: assu v and T}
    \E^v_1(0,\tau):= \rL^2(0,\tau; \rH^{4}(\Omega;\R^2)) \cap \rH^{1}(0,\tau;\rH^2(\Omega;\R^2)) \ \text{ and } \    \E^\Theta_1(0,\tau):= \rL^2(0,\tau; \rH^{4}(\Omega)) \cap \rH^{1}(0,\tau;\rH^2(\Omega)). 
\end{equation*}
Moreover, denote by $\Theta^\ast$ a positive constant and assume that 
\begin{equation*}
    \| v \|_{\E^v_1(0,\tau)} + \| \Theta - \Theta^\ast \|_{\E^\Theta_1(0,\tau)} \leq \eps,
\end{equation*}
for some $\eps>0$. Using the embeddings 
\begin{equation*}
    \E^{\Theta}_1(0,\tau) \hookrightarrow \rL^\infty(0,\tau; \rH^3(\Omega)) \hookrightarrow \rL^\infty(0,\tau; \rW^{1,\infty}(\Omega)),
\end{equation*}
there is a constant $C_{\mathrm{emb}}(\Omega,\tau)>0$ such that
\begin{equation*}
    \| \Theta - \Theta^\ast \|_{\rL^\infty(0,\tau;\rL^\infty(\Omega))} \leq C_{\mathrm{emb}} \| \Theta - \Theta^\ast \|_{\E_1(0,\tau)} \leq C\eps.
\end{equation*}
Choosing $C_{\mathrm{emb}} \cdot \eps < \frac{\Theta^\ast}{2}$ results in 
\begin{equation*}
    \frac{\Theta^\ast}{2} \leq \Theta(t,x,y,z) \leq \frac{3\,\Theta^\ast}{2} \ \text{ for all }\ (t,x,y,z) \in [0,\tau] \times \Omega. 
\end{equation*}
Since the maps
\begin{equation*}
    \Theta \mapsto \Theta^{-1},  \ \int_0^z \Theta^{-1} \ \text{ and } \ \sigma \mapsto \frac{1}{1-\mathrm{exp}(-\sigma)}
\end{equation*}
are smooth in the regime where the temperature is bounded from above and below, standard composition estimates yield that $\hat{B}(\Theta) \in \E_1(0,\tau)$ satisfies 
\begin{equation*}
    \| \hat{B}(\Theta) - \hat{B}(\Theta^\ast)(z) \|_{\E^\Theta_1(0,\tau)} \leq C  \| \Theta - \Theta^\ast \|_{\E^\Theta_1(0,\tau)} \leq C\eps,  
\end{equation*}
where $\hat{B}(\Theta^\ast)(z)$ is the induced state for the functional $\hat{B}(\Theta)$ coming from the constant state $\Theta^\ast$, this is,
\begin{equation*}
    \hat{B}(\Theta^\ast)(z) := \frac{\mathrm{exp}(-z/ \Theta^\ast)}{\Theta^\ast (1- \mathrm{exp}(-1/\Theta^\ast)) }.
\end{equation*}
We note in particular that the state $\hat{B}(\Theta^\ast)$ depends on the vertical variable, which is induced by the equation since $\dz p = \varrhobar \, \dz ( \hat{B}(\Theta) \, \Theta) = - \varrhobar \hat{B}(\Theta)$.  Furthermore, we obtain
\begin{equation*}
    \frac{2}{3 \, \Theta^\ast} \, \frac{\mathrm{exp}(-2 / \Theta^\ast)}{1- \mathrm{exp}(-2 / \Theta^\ast)} \leq \hat{B}(\Theta(t,x,y,z)) \leq \frac{2}{\Theta^\ast} \, \frac{1}{1- \mathrm{exp}(-2 / (3\,\Theta^\ast))} \ \text{ for all }\ (t,x,y,z) \in [0,\tau] \times \Omega.
\end{equation*}
Observe that the space $\rH^{3}(\T^2)$ behaves like a Banach algebra with respect to products, which yields the estimate
\begin{equation*}
    \begin{aligned}\label{eq: est b space}
        \| \nablaH b(t) \|_{\rH^{3}(\T^2)} &\leq \int_0^1 \| \nablaH \hat{B}(\Theta(t,\cdot,\eta)) \|_{\rH^{3}(\T^2)}\, \| v(t,\cdot,\eta) \|_{\rH^{3}(\T^2)} + \| \hat{B}(\Theta(t,\cdot,\eta))\|_{\rH^{3}(\T^2)}\,  \| \nablaH v(t,\cdot,\eta) \|_{\rH^{3}(\T^2)} \d \eta \\
        &\leq  C\| \hat{B}(\Theta(t)) \|_{\rH^{4}(\Omega)} \, \| v(t) \|_{\rH^{3}(\Omega)} + \| \hat{B}(\Theta(t)) \|_{\rH^{3}(\Omega)} \, \|v(t)\|_{\rH^{4}(\Omega)}
    \end{aligned}
\end{equation*}
and therefore
\begin{equation*}
   \begin{aligned}
        \| \nablaH b \|_{\rL^2(0,\tau;\rH^{3}(\T^2))} & \leq  \| \hat{B}(\Theta) \|_{\rL^2(0,\tau;\rH^{4}(\Omega))} \, \| v \|_{\rL^\infty(0,\tau;\rH^{3}(\Omega))} + \| \hat{B}(\Theta) \|_{\rL^\infty(0,\tau;\rH^{3}(\Omega))} \, \|v\|_{\rL^2(0,\tau;\rH^{4}(\Omega))} \\ &\leq C \| \hat{B}(\Theta) \|_{\E^\Theta_1(0,\tau)} \, \| v \|_{\E^v_1{0,\tau)}} 
        \\&\leq C \eps.
   \end{aligned}
\end{equation*}
For $\rX$ as in \eqref{eq:diffeo X Euler Lagrange sea ice para-hyper}, we conclude from H\"older's inequality that
\begin{equation}\label{eq:est of nablaH X - Id sea ice para-hyper}
    \begin{aligned}
        \sup_{t\in [0,\tau]} \| \nablaH \rX - \mathrm{I}_2 \|_{\rH^{3}(\T^2)}
        \le C \int_0^\tau \| \nablaH b(s,\cdot) \|_{\rH^{3}(\T^2)} \d s
        \le C \tau^{\nicefrac{1}{2}}\eps. 
    \end{aligned}
\end{equation}
In particular, we conclude that by choosing $\eps \in (0, \alpha)$ where $\alpha = \min \big ( \frac{1}{2C \tau^{\nicefrac{1}{2}}}, \frac{\Theta^\ast}{2C_{\mathrm{emb}}} \big )$ we obtain the estimate
\begin{equation}\label{eq:est of nablaHX - Id in Linfty}
    \sup_{t \in [0,\tau]} \| \nablaH \rX - \mathrm{I}_2 \|_{\rW^{1,\infty}(\T^2)} \le \frac{1}{2}
\end{equation}
and a Neumann series argument guarantees invertibility of  $\nablaH \rX$. 
Denoting by $\mathrm{Y}(t,\cdot)$ the inverse of $\rX(t,\cdot)$ we find that
\begin{equation*}
    \nablaH \mathrm{Y}(t,\rX(t,y_\H)) = [\nablaH \rX]^{-1}(t,y_\H). 
\end{equation*}
Using $\dt \nablaH \rX(t,\cdot) = \nablaH b(t,\cdot)$, we obtain 
\begin{equation}\label{eq:timederX}
   \begin{aligned}
        \quad\| \dt \nablaH \rX(t,\cdot) \|_{\rL^2(0,\tau;\rH^{3}(\T^2))} \le  C \eps,
   \end{aligned}
\end{equation}
and therefore there is a constant $C > 0$ such that
\begin{equation}\label{eq:est of nablaH X in W1p and Linfty}
    \| \nablaH \rX \|_{\rH^{1}(0,\tau;\rH^{3}(\T^2))} + \| \nablaH \rX \|_{\rL^\infty(0,\tau;\rH^{3}(\T^2))} \le C.
\end{equation}
Since $\rH^{1}(0,\tau;\rH^{3}(\T^2)) \cap\rL^\infty(0,\tau;\rH^{3}(\T^2))$ is a Banach algebras with respect to point wise multiplication.
As a consequence of \eqref{eq:est of nablaH X in W1p and Linfty} we deduce the existence of a constant $C > 0$ such that 
\begin{equation*}
    \begin{aligned}
        \| \det \nablaH \rX \|_{\rH^{1}(0,\tau;\rH^{3}(\T^2))} + \| \det \nablaH \rX \|_{\rL^\infty(0,\tau;\rH^{3}(\T^2))} 
        &\le C \text{ and }\\
        \| \Cof\, \nablaH \rX \|_{\rH^{1}(0,\tau;\rH^{3}(\T^2))} + \| \Cof\, \nablaH \rX \|_{\rL^\infty(0,\tau;\rH^{3}(\T^2))} 
        &\le C.
    \end{aligned}
\end{equation*}
Thanks to \eqref{eq:est of nablaHX - Id in Linfty}, we find that $\det \nablaH \rX \ge C > 0$ on $(0,\tau) \times \T^2$ for some constant $C > 0$ provided $\eps \in (0,\alpha)$.
The representation
\begin{equation}\label{eq:ex of nablaH Y sea ice para-hyper}
    \mathrm{Z} := [\nablaH \rX]^{-1} = \frac{1}{\det \nablaH \rX} (\Cof \nablaH \rX)^\top
\end{equation}
then results in
\begin{equation*}
    \|  \mathrm{Z} \|_{\rH^{1}(0,\tau;\rH^{3}(\T^2))} + \| \mathrm{Z} \|_{\rL^\infty(0,\tau;\rH^{3}(\T^2))}\ \le C \  \text{ and }\ \| \mathrm{Z} - \mathrm{I}_2\|_{\rL^\infty(0,\tau;\rH^{3}(\T^2))} \leq C\eps.
\end{equation*}
The above properties of the Lagrangian transformation $\rX$ are summarized in the following.
\begin{lem}\label{lem:ests of trafo}
Let $\tau>0$ and assume that $(v, \Theta) \in \E^v_1(0,\tau) \times \E^\Theta_1(0,\tau)$ and 
\begin{equation*}
        \| v \|_{\E_1(0,\tau)} + \| \Theta - \Theta^\ast \|_{\E_1(0,\tau)} \leq \eps,
\end{equation*}
for some constants $\Theta^\ast, \eps = \eps(\tau, \Theta^\ast, \Omega)>0$.
  Denote by $\rX$ the flow associated to the field $b$ as made precise in \eqref{eq:flow}.

\begin{enumerate}[(a)]
    \item  There is a constant $C > 0$ such that
    \begin{equation*}
        \sup_{t \in (0,\tau
    )} \| \nabla \rX(t,\cdot) - \mathrm{I}_2 \|_{\rW^{1,\infty}(\T^2)} + \sup_{t \in (0,\tau)} \| \nabla \rX(t,\cdot) - \mathrm{I}_2 \|_{\rH^{3}(\T^2)} \le C \eps
    \end{equation*}
    If $\eps\in (0,\alpha)$, we especially have
    \begin{equation*}
        \| \nabla \rX(t,\cdot) - \mathrm{I}_2 \|_{\rL^\infty(0,\tau;\rW^{1,\infty}(\T^2))} \le \frac{1}{2},
    \end{equation*}
    so $\nabla \rX(t,\cdot)$ is invertible, and $\mathrm{Z} = [\nabla \rX]^{-1}$ is thus well-defined.
    \item Similar estimates as in (a) are also valid for $\mathrm{Z}$ as well as $\mathrm{Z}^\top$, i.\ e.,
    \begin{equation*}
        \begin{aligned}
            \sup_{t \in (0,\tau)} \|\mathrm{Z}(t,\cdot) - \mathrm{I}_2 \|_{\rH^{3}(\T^2)} + \sup_{t \in (0,\tau)} \| \mathrm{Z}^\top(t,\cdot) - \mathrm{I}_2 \|_{\rH^{3}(\T^2)} 
            &\le C \eps \ \text{ and }\\
            \| \mathrm{Z} \|_{\rL^\infty(0,\tau;\rH^{3}(\T^2))} + \| \mathrm{Z}^\top \|_{\rL^\infty(0,\tau;\rH^{3}(\T^2))} 
            &\le C
        \end{aligned}
    \end{equation*}
    holds true for some constant $C = C(\T^2,\tau) > 0$.
    % \begin{equation*}
    %     \sup_{t \in (0,\tau)} \| Z(t,\cdot) - \mathrm{Id}_3 \|_{\rH_x^{1,q}} + \sup_{t \in (0,\tau)} \| Z^\top(t,\cdot) - \mathrm{Id}_3 \|_{\rH_x^{1,q}} \le C \eps^2, \| Z \|_{\rL_\tau^\infty(\rH_x^{1,q})} \le C,  \| Z^\top \|_{\rL_\tau^\infty(\rH_x^{1,q})} \le C
    % \end{equation*}
    \item For all $j$, $k$, $l = 1,2$, there is $C = C(\T^2,\tau) > 0$ with
    \begin{equation*}
        \Bigl\| \frac{\partial \mathrm{Z}_{l,j}}{\partial y_k} \Bigr\|_{\rL^\infty(0,\tau;\rH^2(\T^2))} \le C \eps.
    \end{equation*}
    %\item there is a constant $C = C(p,q,\Omega,\tau) > 0$ so that the estimate
   % \begin{equation*}
     %   \| Z - \Id_3 \|_{\rC_\tau^{0,\nicefrac{1}{p'}}(\rH_x^{1,q})} \le C \eps
    %\end{equation*}
    %is valid, where $p' \in (1,2)$ is the H\"older conjugate of $p$, i.\ e., $\nicefrac{1}{p}+\nicefrac{1}{p'}=1$.
\end{enumerate}
\end{lem}
\noindent
With the estimate of the transformation at hand, we introduce the new unknowns following the characteristics $\rX$. For $ \varrhobar^\ast, \Theta^\ast>0$ define
\begin{equation*}
    \varrhobar^\rL(t,\cdot) := \varrhobar(t,\rX(t,\cdot)) - \varrhobar^\ast, \quad v^\rL(t,\cdot,z) := v(t,\rX(t,\cdot),z) , \ \text{ and }\ \Theta^\rL(t,\cdot,z):= \Theta(t,\rX(t,\cdot),z)- \Theta^\ast.
\end{equation*}
The transformed functional $\hat{B}^\rL$ is then defined diagnostically by 
\begin{equation*}
  \hat{B}^\rL(t,\cdot,z) := \hat{B}(\Theta(t,\rX(t,\cdot),z))= \hat{B}\big (\Theta^\rL(t,\cdot,z) + \Theta^\ast \big ) =\hat{B}(\Theta^\ast) + \delta \hat{B} (\Theta^\rL, \Theta^\ast  )(t,\cdot,z),
\end{equation*}
where we set
\begin{equation*}
    \delta \hat{B}(\Theta^\rL,  \Theta^\ast  ) (t,\cdot, z):= \hat{B}  ( \Theta^\rL (t,\cdot,z)+ \Theta^\ast ) -\hat{B}(\Theta^\ast).
\end{equation*}
Note that  the map $\Theta \mapsto \hat{B}(\Theta)$ is composed of smooth scalar functions and bounded linear operators, all of which are Fr\'echet differentiable in the regime where $\Theta$ is uniformly bounded away from $0$. Denoting by $D\hat{B}$ the Fr\'echet derivative and using that $\hat{B}(\Theta^\ast)$ only depends on the vertical variable, we obtain
\begin{equation*}
    \nablaH \hat{B} ( \Theta^\rL + \Theta^\ast) = (D\hat{B})(\Theta^\rL + \Theta^\ast) \cdot \nablaH \Theta^\rL = \big [(D\hat{B})(\Theta^\ast)+ \delta (D \hat{B})(\Theta^\rL,\Theta^\ast)\big] ( \nablaH \Theta^\rL),
\end{equation*}
setting similar to above 
\begin{equation*}
    \delta( D \hat{B})(\Theta^\rL,\Theta^\ast) := (D\hat{B})(\Theta^\rL+\Theta^\ast) - D \hat{B}(\Theta^\ast).
\end{equation*}
Note  that the map $\Theta \mapsto \hat{B}(\Theta)$ is locally Lipschitz on the space $\E^\Theta_1 (0,\tau)$ in the regime where 
\begin{equation}\label{eq: regime}
    \frac{\Theta^\ast}{2} \leq \Theta\leq \frac{3\, \Theta^\ast}{2} \ \text{ for all }\ (t,x,y,z) \in [0,\tau] \times \Omega. 
\end{equation}
The Lipschitz estimate then yields
\begin{equation*}
   \| \delta \hat{B}(\Theta^\rL,  \Theta^\ast  ) \|_{\E^\Theta_1(0,\tau)} = \| \hat{B}  ( \Theta^\rL + \Theta^\ast ) - \hat{B}(\Theta^\ast) \|_{\E^\Theta_1(0,\tau)} \leq C \| \Theta^\rL \|_{\E^\Theta_1(0,\tau)},
\end{equation*}
for a constant $C=C(\Theta^\ast)>0$.
To obtain an analogous bound for the operator $\delta (D \hat{B})$ we first calculate the Fr\'echet derivative $(D\hat{B})(\Theta)$ explicitly. Setting
\begin{equation*}
    A(\Theta)(z) := \int_0^z \frac{1}{\Theta(\cdot, \eta)} \d \eta, \quad I(\Theta) := 1- \mathrm{exp}\big (-A(\Theta)(1) \big ) \ \text{ and }\ N(\Theta)(z):= \frac{\mathrm{exp}\big ( -A(\Theta)(z)\big )}{\Theta(z)}
\end{equation*}
so that $\hat{B}(\Theta) = N(\Theta)(z) / I(\Theta)$, a direct differentiation gives, for any direction $h=h(z)$
\begin{equation*}
  \begin{aligned}
        (D\hat{B})(\Theta)[h](z) &= \frac{\mathrm{exp}\big ( -A(\Theta)(z)\big )}{I(\Theta)} \Big ( - \frac{h(z)}{\Theta^2} + \frac{1}{\Theta} \int_0^z \frac{h(\eta)}{\Theta(\cdot, \eta)^2} \d \eta \Big) \\&\quad+ \frac{N(\Theta)(z)}{I(\Theta)^2} \,\mathrm{exp}\big (-A(\Theta)(1) \big ) \int_0^1 \frac{h(\eta)}{\Theta(\cdot, \eta)^2} \d \eta.
  \end{aligned}
\end{equation*}
So $ (D\hat{B})(\Theta)$ is a sum of multiplication operator and two $z$-integrals. Moreover, the linearized operator $(D\hat{B})(\Theta^\ast)$ is explicitly given by 
\begin{equation}\label{eq: DB lin}
    (D\hat{B})(\Theta^\ast)[h](z) = \hat{B}(\Theta^\ast) \Big ( - \frac{h(z)}{\Theta^\ast} + \frac{1}{(\Theta^\ast)^2} \int_0^z h(\eta) d\eta + \frac{\mathrm{exp}(-1/\Theta^\ast)}{(\Theta^\ast)^2 (1- \mathrm{exp}(-1/\Theta^\ast))} \int_0^1 h(\eta) \d \eta \Big ).
\end{equation}
Notice that the maps
\begin{equation*}
    \Theta \mapsto \Theta^{-1}, \ \Theta^{-2}, \ \mathrm{exp}(-A(\Theta)(\cdot)) , \ I(\Theta)^{-1} \ \text{ and } \ I(\Theta)^{-2} 
\end{equation*}
are all locally Lipschitz on $\E_1(0,\tau)$ in the regime \eqref{eq: regime} and multiplication and integral operators are bounded, which means that $(D\hat{B})(\Theta)$ is bounded. In particular, we obtain 
\begin{equation*}
    \| \delta (D \hat{B})(\Theta^\rL,\Theta^\ast) \|_{\mathrm{op}} = \| (D\hat{B} ) ( \Theta^\rL + \Theta^\ast ) - (D\hat{B})(\Theta^\ast) \|_{\mathrm{op}} \leq C \| \Theta^\rL \|_{\E^\Theta_1(0,\tau)},
\end{equation*}
for a constant $C=C(\Theta^\ast)>0$.
With the above convention, we write
\begin{equation*}
    b^\rL(t,\cdot):= b(t,\rX(t,\cdot)) = \int_0^1\hat{B}(\Theta^\ast) \, v^\rL(t,\cdot,z) \d z + \int_0^1    \delta \hat{B}(\Theta^\rL,  \Theta^\ast  ) (t,\cdot, z) \, v^\rL(t,\cdot,z) \d z
\end{equation*}
and 
\begin{equation}\label{eq: divH b}
    \divH b^\rL  = \int_0^1 \Big [ \Big(\big ((D\hat{B})(\Theta^\ast) + \delta ( D\hat{B})(\Theta^\rL, \Theta^\ast)\big ) (\nablaH \Theta^\rL) \Big) \,v^\rL + \big ( \hat{B}(\Theta^\ast) + \delta \hat{B}( \Theta^\rL, \Theta^\ast) \big ) \,\divH v^\rL \Big ](\cdot,z) \d z
\end{equation}
Finally, the transformed vertical velocity $w^\rL(t,\cdot,z):= w(t,\rX(t,\cdot),z)$ is given by transforming the non-local formula \eqref{eq:W-noQt}. First, notice that
\begin{equation*}
    \dt \hat{B} + \divH (\hat{B} \,v ) - \hat{B}\, \divH b = \dt \hat{B} + b \cdot \nablaH \hat{B} + \divH \big ( \hat{B}\, ( v-b) \big )
\end{equation*}
and therefore 
\begin{equation}\label{eq:wL_rhobar_hatB}
\begin{aligned}
(\varrhobar \, \hat{B} \, w)^{\rL}
&=
-
\int_{0}^{z}
\Big[
\hat{B}^\rL ( v^\rL - b^\rL) \cdot \mathrm{Z}^\top \nablaH \varrhobar^\rL + (\varrhobar^\rL+ \varrhobar^\ast) \big ( \dt \hat{B}^\rL + \nablaH \big ( \hat{B}^\rL ( v^\rL - b^\rL) \big ) : \mathrm{Z}^\top 
\Big](\cdot, \eta)\,d\eta \\
&= -\int_0^z \Big [ \big (\hat{B}(\Theta^\ast) + \delta \hat{B}(\Theta^\rL, \Theta^\ast)\big ) \Big ( (v^\rL - b^\rL ) \cdot \mathrm{Z}^\top \nablaH \varrhobar^\rL  + (\varrhobar^\rL+\varrhobar^\ast) \,\nablaH (v^\rL - b^\rL ) : \mathrm{Z}^\top \Big ) \\ & \quad   +\Big (\varrhobar^\rL+\varrhobar^\ast \Big) \Big ((D\hat{B})(\Theta^\ast) + \delta ( D\hat{B})(\Theta^\rL, \Theta^\ast) \Big ) \,\Big ( \dt \Theta^\rL + \big ( \mathrm{Z}\,(v^\rL - b^\rL ) \big ) \cdot \nablaH \Theta^\rL \Big)\Big ](\cdot ,\eta) \d \eta.
\end{aligned}
\end{equation}
For the transformation of the vertical pressure work $p \,\dz w$ in the temperature equation, it is convenient to split the above diagnostic formulation of $(\varrhobar \, \hat{B} \, w)^{\rL}$ into two parts $(\varrhobar \, \hat{B} \, w)^{\rL} = J_1 +J_2$ of which the first $J_1$ collects terms of order one with respect to the solution and $J_2$ collects higher orders. Explicitly, we find that $J_1$ and $J_2$ are given by
\begin{equation}
    \begin{aligned}
        J_1 &= -\varrhobar^\ast \int_0^z \big [(D\hat{B})(\Theta^\ast)   (\dt \Theta^\rL) + \hat{B}(\Theta^\ast) \divH v^\rL  \big ] (\cdot , \eta) \d \eta \\ &\quad +\varrhobar^\ast \Big(\int_0^z \hat{B}(\Theta^\ast)(\cdot,\eta)\,\d\eta\Big)
\Big(\int_0^1[ \hat{B}(\Theta^\ast)\, \divH v^\rL\big ](t,\cdot,\zeta)\,\d\zeta\Big), \\
        J_2 &=  -\int_0^z \Big [ \big (\hat{B}(\Theta^\ast) + \delta \hat{B}(\Theta^\rL, \Theta^\ast)\big ) \Big ( (v^\rL - b^\rL ) \cdot \mathrm{Z}^\top \nablaH \varrhobar^\rL  + (\varrhobar^\rL+\varrhobar^\ast) \,\nablaH (v^\rL - b^\rL ) : \big (\mathrm{Z}^\top -\mathrm{I}_2 \big )\Big )\\&\quad +\hat{B}(\Theta^\ast)     \varrhobar^\rL \,\divH (v^\rL - b^\rL ) + \delta \hat{B}(\Theta^\rL, \Theta^\ast)  \varrhobar^\rL    \,\divH (v^\rL - b^\rL ) 
        + \delta \hat{B}(\Theta^\rL, \Theta^\ast)\varrhobar^\ast \,\divH (v^\rL - b^\rL ) \\ &\quad +\hat{B}(\Theta^\ast)     \varrhobar^\ast \Big (\divH v^\rL -\Big(\big ((D\hat{B})(\Theta^\ast) + \delta ( D\hat{B})(\Theta^\rL, \Theta^\ast)\big ) (\nablaH \Theta^\rL) \Big) \,v^\rL + \big (  \delta \hat{B}( \Theta^\rL, \Theta^\ast) \big ) \,\nablaH v^\rL \Big )
        \\ & \quad   +\varrhobar^\rL \Big ((D\hat{B})(\Theta^\ast) + \delta ( D\hat{B})(\Theta^\rL, \Theta^\ast) \Big ) \,\Big ( \dt \Theta^\rL + \big ( \mathrm{Z}\,(v^\rL - b^\rL ) \big ) \cdot \nablaH \Theta^\rL \Big)\\ &\quad +\varrhobar^\ast \Big ((D\hat{B})(\Theta^\ast) + \delta ( D\hat{B})(\Theta^\rL, \Theta^\ast) \Big ) \,\Big (  \mathrm{Z}\,(v^\rL - b^\rL ) \cdot \nablaH \Theta^\rL \Big)
        \Big ](\cdot ,\eta) \d \eta 
    \end{aligned}\label{eq J_1 and J_2}
\end{equation}
To transform the vertical pressure work $p \, \dz w$, we first observe that by hydrostatic balance the following identity holds
\begin{equation*}
    p \, \dz w = \dz (p w) + \varrho w = (\dz \Theta + 1) \varrho w + \Theta \, \dz (\varrho w)
\end{equation*}
and therefore 
\begin{equation*}
   \begin{aligned}
        p \, \dz w &\rightsquigarrow (\dz \Theta^\rL + 1) (J_1 + J_2) + (\Theta^\rL + \Theta^\ast) \dz (J_1 + J_2) \\&= J_1 + \Theta^\ast \dz J_1 + \dz \Theta^\rL J_1+ (\dz \Theta^\rL + \Theta^\ast + 1) J_2 +\Theta^\ast \dz J_2 + \Theta^\rL \dz J_1. 
   \end{aligned}
\end{equation*}
Note that only the first two addends $ J_1 + \Theta^\ast \dz J_1 $ are of order one in terms of the solution and therefore must be accounted for in the the left-hand side. Using the formula for $J_1$ from \eqref{eq J_1 and J_2}, we see that there is a nonlocal contribution to the time derivative given by the operator
\begin{equation*}
    -\varrhobar^\ast \Big ( \Theta^\ast (D\hat{B})(\Theta^\ast)(\cdot ) + \int_0^z (D\hat{B})(\Theta^\ast) (\cdot ) \d \eta \Big ) = \varrhobar^\ast \Big ( \hat{B}(\Theta^\ast) (\cdot ) - \frac{\mathrm{exp}(-1/\Theta^\ast)}{(\Theta^\ast)^2 (1- \mathrm{exp}(-1/\Theta^\ast))^2} \int_0^1 (\cdot) \d \eta \Big),
\end{equation*}
where we used \eqref{eq: DB lin}. With the above consideration, we are now in a position to formulate the transformed system which is given by
\begin{equation}
	\left\{
	\begin{aligned}
        \dt \varrhobar^\rL + \bar{\varrho}^\ast \int_0^1 \big[\hat{B}(\Theta^\ast) \, \divH v^\rL\big](\cdot, \eta ) \d \eta  &=f_1, &&\text{ in }\T^2_\tau, \\
\dt v^\rL - \frac{\mu\, \Delta v^\rL}{\varrhobar^\ast \,\hat{B}(\Theta^\ast)} - \frac{\mu'\, \nablaH \divH v^\rL}{\varrhobar^\ast \,\hat{B}(\Theta^\ast)}	+ \Theta^\ast\frac{\nablaH \varrhobar^\rL}{\varrhobar^\ast} +\Big ( \Theta^\ast\frac{ (D\hat{B})(\Theta^\ast)}{\hat{B}(\Theta^\ast)} + \mathrm{I}_2 \Big ) (\nablaH \Theta^\rL) &= f_2, &&\text{ in } \Omega_\tau,  \\
%		  \dz \zeta  &= 0  , &&\text{ on } \Omegaair \times (0,T), \\ %\frac{1}{p} , &&\text{ on } \Omegaair \times (0,T),\\
%		\div \uair &= 0, &&\text{ on } \Omegaair \times (0,T),\\
%        \bigl (\partial_p \vair,\omegaair \bigr ) &= (0, 0), &&\text{ on } \Gammaairu \times (0,T), \\
%        \omegaair  &= 0, &&\text{ on } \Gammaairb \times (0,T), \\
%         \partial_p \vair  &= p_s^{-1} (\vocn - \vair) |\vocn - \vair|, &&\text{ on } \Gammaairb \times (0,T), \\ %\frac{(\vocn - \vair) |\vocn - \vair|}{p_s}, &&\text{ on } \Gammaairb \times (0,T), \\
        %p &=  \rho^\gamma , &&\text{ on } \Omega \times (0,T),  \\
        %\dt \vocn - \Delta \vocn + \uocn \cdot \nabla \vocn + \nablaH \pi &= f^{\ocn}, &&\text{ on } \Omegaocn \times (0,T), \\
        %\dz \pi &= - \rhoocn g, &&\text{ on } \Omegaocn \times (0,T), \\
        %\div \uocn &= 0, &&\text{ on } \Omegaocn \times (0,T). \\
%\dz        V^\air &=0, &&\text{ on } \Gammaairu \times (0,T), \\
     %   \dz V^\air &= B_1(V^\air, V^\ocn) , &&\text{ on } \Gammaairb \times (0,T), \\
       \mathcal{L}[\dt \Theta^\rL]  - \frac{\Delta \Theta^\rL}{\varrhobar^\ast \,\hat{B}(\Theta^\ast)}  \\ - \frac{ \int_0^z \hat{B}(\Theta^\ast) \divH v^\rL(\cdot,\eta) \d \eta}{\hat{B}(\Theta^\ast)}+\Theta^\ast \mathrm{exp}(z/\Theta^\ast)  
\int_0^1 \big[\hat{B}(\Theta^\ast) \, \divH v^\rL\big](\cdot, \eta ) \d \eta &= f_3, &&\text{ in } \Omega_\tau,  \\
		  %\dz \pi  &= 1, &&\text{ on } \Omegaocn \times (0,T), \\ %\frac{1}{p} , &&\text{ on } \Omegaair \times (0,T),\\
		%\div_\H  \bar{V}^{\ocn} &=F_4(V^\ocn), &&\text{ in } G \times (0,T),\\
     %  \dz  V^\ocn &= 0 &&\text{ on } \Gab \times (0,T), \\
     %   \dz V^\ocn &= B_2(\zeta, V^\air,V^\ocn), &&\text{ on } \Gau \times (0,T), \\
%        , \\
%        (\dz \vocn, \wocn) &=(0, 0), &&\text{ on } \Gamma_b \times (0,T), \\
     \varrhobar^\rL(0) = \varrhobar_0- \varrhobar^\ast, \ v^\rL(0) = v_0 \ \Theta^\rL(0) = \Theta_0 - \Theta^\ast &
	\end{aligned}
	\right. 
	\label{eq: full CPE Lagrange}
\end{equation}
where the nonlocal operator $\mathcal{L}$ is given by 
\begin{equation*}
    \mathcal{L}[h] := 2 h -  \frac{\mathrm{exp}\big ( (z-1)/\Theta^\ast\big )}{\Theta^\ast ( 1 - \mathrm{exp}(-1/\Theta^\ast))} \int_0^1 h(\cdot ,\eta) \d \eta.
\end{equation*}
Furthermore, the nonlinear remainders $(f_1, f_2,f_3)(\varrhobar^\rL,v^\rL, \Theta^\rL)$  are given by 
\begin{equation*}
    \begin{aligned}
         f_1(\varrhobar^\rL, v^\rL, \Theta^\rL) &=- \int_0^1 \Big [ \Big(\big (D\hat{B} + \delta ( D\hat{B}) \big ) ( \nablaH \Theta^\rL) \Big) \,v^\rL +  \delta \hat{B}  \,\divH v^\rL \Big ](\cdot,\eta) \d \eta \\&\quad  -\varrhobar^\rL \divH b^\rL - (\varrhobar^\rL + \varrhobar^\ast) \nablaH b^\rL : \left (\mathrm{Z}^\top  - \mathrm{I}_2 \right ),
    \end{aligned}
\end{equation*}
and
\begin{equation*}
    \begin{aligned}
        &\quad \big (f_2(\varrhobar^\rL, v^\rL, \Theta^\rL) \big )_i\\ &= -\Big (\frac{\varrhobar^\rL \,\delta\hat{B}
        }{\varrhobar^\ast \,\hat{B}(\Theta^\ast)} + \frac{\varrhobar^\rL}{\varrhobar^\ast} + \frac{\delta\hat{B}} {\hat{B}(\Theta^\ast)}\Big )\, \dt v_i^\rL - \frac{(\varrhobar^\rL + \varrhobar^\ast) ( \hat{B}(\Theta^\ast) +\delta \hat{B}) \, \big ( \mathrm{Z}(v^\rL - b^\rL) \cdot \nablaH \big ) v^\rL}{\varrhobar^\ast \,\hat{B}(\Theta^\ast)}\\& \quad + \frac{\mu}{\varrhobar^\ast \,\hat{B}(\Theta^\ast)}\Bigl ( \sum_{j,k,l} \frac{\partial^2 v_i^\rL}{\partial y_k \partial y_l} \bigl (\mathrm{Z}_{k,j} - \delta_{k,j}\bigr ) \mathrm{Z}_{l,j} + \sum_{k,l}\frac{\partial^2 v_i^\rL}{\partial y_k \partial y_l}\bigl (\mathrm{Z}_{l,k}- \delta_{l,k}\bigr ) +  \sum_{j,k,l} \mathrm{Z}_{l,j} \frac{\partial v_i^\rL}{\partial y_k} \frac{\partial \mathrm{Z}_{k,j}}{\partial y_l} \Bigr ) \\ &\quad  +\frac{\mu'}{\varrhobar^\ast \,\hat{B}(\Theta^\ast)}
        \Bigl (\sum_{l,k,j} \frac{\partial^2 v^\rL_j}{\partial y_k \partial y_l}\bigl ( \mathrm{Z}_{l,i}- \delta_{l,i}\bigr )\mathrm{Z}_{k,j}  +  \sum_{k,l}\frac{\partial^2 v^\rL_l}{\partial y_k \partial y_l} \bigl ( \mathrm{Z}_{l,i}- \delta_{l,i}\bigr ) +  \sum_{j,k,l} \mathrm{Z}_{l,i} \frac{\partial v^\rL_j}{\partial y_k} \frac{\partial \mathrm{Z}_{k,j}}{\partial y_l} \Bigr ) \\
        &\quad - \frac{\Theta^\ast \, \big ( ( \mathrm{Z} - \mathrm{I}_2 )\nablaH \varrhobar^\rL\big )_i}{\varrhobar^\ast} - \frac{\delta\hat{B}\, \Theta^\rL \, \big (\mathrm{Z}^\top \nablaH \varrhobar^\rL\big)_i}{\varrhobar^\ast \,\hat{B}(\Theta^\ast)} - \frac{\delta\hat{B}\, \Theta^\ast \,  \big (\mathrm{Z}^\top \nablaH \varrhobar^\rL\big)_i}{ \varrhobar^\ast \,\hat{B}(\Theta^\ast)} - \frac{\Theta^\rL \, \big (\mathrm{Z}^\top \nablaH \varrhobar^\rL\big)_i}{\varrhobar^\ast}  \\&\quad  -\Big ((D\hat{B})(\Theta^\ast)+ \delta (D \hat{B})\Big)\Big ( (\mathrm{Z}-\mathrm{I}_2) \nablaH \Theta^\rL + \frac{\varrhobar^\rL \, \Theta^\rL \, \mathrm{Z}^\top  \nablaH \Theta^\rL}{\varrhobar^\ast \,\hat{B}(\Theta^\ast)} +\frac{\varrhobar^\rL\, \Theta^\ast \, \mathrm{Z}^\top \nablaH \Theta^\rL}{\varrhobar^\ast \,\hat{B}(\Theta^\ast)} + \frac{\Theta^\rL \, \mathrm{Z}^\top \nablaH \Theta^\rL }{\hat{B}(\Theta^\ast)} \Big )_i
        \\ &\quad -\delta \hat{B} (\Theta^\rL, \Theta^\ast  ) ( \nablaH \Theta^\rL)_i - (\mathrm{Z}-\mathrm{I}_2) \nablaH \Theta^\rL - \frac{\varrhobar^\rL \, \delta \hat{B} \, (\mathrm{Z}^\top \nablaH \Theta^\rL)_i}{\varrhobar^\ast \,\hat{B}(\Theta^\ast)}- \frac{\varrhobar^\rL\,(\mathrm{Z}^\top \nablaH \Theta^\rL)_i}{\varrhobar^\ast} - \frac{\delta\hat{B}\, (\mathrm{Z}^\top \nablaH \Theta^\rL)_i}{\hat{B}(\Theta^\ast)} \\ 
        &\quad - \frac{(\varrhobar \, \hat{B} \, w)^{\rL} \, (\dz v^\rL)_i}{\varrhobar^\ast \,\hat{B}(\Theta^\ast)}
    \end{aligned}
\end{equation*}
as well as 
\begin{equation*}
    \begin{aligned}
         &\quad f_3(\varrhobar^\rL, v^\rL, \Theta^\rL) \\ &= -\Big (\frac{\varrhobar^\rL \,\delta\hat{B}
        }{\varrhobar^\ast \,\hat{B}(\Theta^\ast)} + \frac{\varrhobar^\rL}{\varrhobar^\ast} + \frac{\delta\hat{B}} {\hat{B}(\Theta^\ast)}\Big )\, \dt \Theta^\rL-\frac{(\varrhobar^\rL + \varrhobar^\ast) ( \hat{B}(\Theta^\ast) +\delta \hat{B}) \, \big ( \mathrm{Z}(v^\rL - b^\rL) \cdot \nablaH \big ) \Theta^\rL}{\varrhobar^\ast \,\hat{B}(\Theta^\ast)} \\& \quad + \frac{1}{\varrhobar^\ast \,\hat{B}(\Theta^\ast)}\Bigl ( \sum_{j,k,l} \frac{\partial^2 \Theta^\rL}{\partial y_k \partial y_l} \bigl (\mathrm{Z}_{k,j} - \delta_{k,j}\bigr ) \mathrm{Z}_{l,j} + \sum_{k,l}\frac{\partial^2 \Theta^\rL}{\partial y_k \partial y_l}\bigl (\mathrm{Z}_{l,k}- \delta_{l,k}\bigr ) +  \sum_{j,k,l} \mathrm{Z}_{l,j} \frac{\partial \Theta^\rL}{\partial y_k} \frac{\partial \mathrm{Z}_{k,j}}{\partial y_l} \Bigr )
         \\ 
        &\quad - \frac{(\varrhobar \, \hat{B} \, w)^{\rL} \, (\dz \Theta^\rL)}{\varrhobar^\ast \,\hat{B}(\Theta^\ast)} - \frac{ \dz \Theta^\rL J_1+ (\dz \Theta^\rL + \Theta^\ast + 1) J_2 +\Theta^\ast \dz J_2 + \Theta^\rL \dz J_1}{\varrhobar^\ast \,\hat{B}(\Theta^\ast)} -\Theta^\ast \nablaH v^\rL : \big ( \mathrm{Z}^\top - \mathrm{I}_2 \big ) \\&\quad- \Big ( \frac{\varrhobar^\rL \, \Theta^\rL}{\varrhobar^\ast} + \frac{\delta \hat{B}\, \varrhobar^\rL \, \Theta^\rL}{\varrhobar^\ast\, \hat{B}(\Theta^\ast)} +\Theta^\rL + \frac{\delta \hat{B}\, \Theta^\ast}{\varrhobar^\ast}+ \frac{\delta \hat{B}\, \Theta^\rL}{\hat{B}(\Theta^\ast)} + \frac{\delta \hat{B}\, \varrhobar^\rL \, \Theta^\ast}{\varrhobar^\ast\, \hat{B}(\Theta^\ast)} +\frac{\varrhobar^\rL\, \Theta^\ast}{\varrhobar^\ast} \Big ) \nablaH v^\rL :\mathrm{Z}^\top
     \end{aligned}
\end{equation*}
Finally, we note that the boundary conditions \eqref{eq: bc} are unaffected by the transformation as it only acts in horizontal direction.

\section{Linear Theory and nonlinear estimates}\label{sec: local}
\noindent
In this section, we study the linearized system \eqref{eq: full CPE Lagrange} that results from the transformation discussed in \autoref{sec: Lagrange} and estimate the nonlinear remainder terms $(f_1,f_2,f_3)$. We introduce some notation that will be used throughout the rest of this section.
Define the functions $\alpha$ and $\beta$ by 
\begin{equation*}
    \alpha(z) := \big (\varrhobar^\ast \, \hat{B}(\Theta^\ast)(z) \big )^{-1}  \ \text{ and }\ \beta(z) :=\frac{\mathrm{exp}\big ( (z-1)/\Theta^\ast\big )}{\Theta^\ast ( 1 - \mathrm{exp}(-1/\Theta^\ast))} 
\end{equation*}
as well as the operator $\mathcal{I}_z, \mathcal{A}$ and $\mathcal{L}$ for a regular function $f$ by 
\begin{equation*}
    \mathcal{I}_z f := \int_0^z \big [ \hat{B}(\Theta^\ast) f \big ](\cdot, \eta) \d \eta ,\quad  \mathcal{A}f := \Big ( \Theta^\ast\frac{ (D\hat{B})(\Theta^\ast)}{\hat{B}(\Theta^\ast)} + \mathrm{I}_2\Big )\nablaH f \ \text{ and }\   \mathcal{L}f = 2f - \beta(z) \int_0^1 f(\cdot, \eta)\d \eta .
\end{equation*}
Using these notation, we consider the fully linearized problem associated to \eqref{eq: full CPE Lagrange} given by
\begin{equation}
	\left\{
	\begin{aligned}
        \dt \xi + \bar{\varrho}^\ast\,  \mathcal{I}_1 ( \divH V)  &=g_1, &&\text{ in }\T^2_\tau, \\
\dt V - \mu\, \alpha\, \Delta V - \mu'\, \alpha\, \nablaH \divH V + \nablaH \xi+ \mathcal{A}T &= g_2, &&\text{ in } \Omega_\tau,  \\
%		  \dz \zeta  &= 0  , &&\text{ on } \Omegaair \times (0,T), \\ %\frac{1}{p} , &&\text{ on } \Omegaair \times (0,T),\\
%		\div \uair &= 0, &&\text{ on } \Omegaair \times (0,T),\\
%        \bigl (\partial_p \vair,\omegaair \bigr ) &= (0, 0), &&\text{ on } \Gammaairu \times (0,T), \\
%        \omegaair  &= 0, &&\text{ on } \Gammaairb \times (0,T), \\
%         \partial_p \vair  &= p_s^{-1} (\vocn - \vair) |\vocn - \vair|, &&\text{ on } \Gammaairb \times (0,T), \\ %\frac{(\vocn - \vair) |\vocn - \vair|}{p_s}, &&\text{ on } \Gammaairb \times (0,T), \\
        %p &=  \rho^\beta , &&\text{ on } \Omega \times (0,T),  \\
        %\dt \vocn - \Delta \vocn + \uocn \cdot \nabla \vocn + \nablaH \pi &= f^{\ocn}, &&\text{ on } \Omegaocn \times (0,T), \\
        %\dz \pi &= - \rhoocn g, &&\text{ on } \Omegaocn \times (0,T), \\
        %\div \uocn &= 0, &&\text{ on } \Omegaocn \times (0,T). \\
%\dz        V^\air &=0, &&\text{ on } \Gammaairu \times (0,T), \\
     %   \dz V^\air &= B_1(V^\air, V^\ocn) , &&\text{ on } \Gammaairb \times (0,T), \\
       \mathcal{L}[\dt T]  - \alpha \,\Delta T  - \varrhobar^\ast\alpha \, \mathcal{I}_z ( \divH V)+\Theta^\ast \mathrm{exp}(z/\Theta^\ast)  \;
\mathcal{I}_1 (\divH V) &= g_3, &&\text{ in } \Omega_\tau,  \\
		  %\dz \pi  &= 1, &&\text{ on } \Omegaocn \times (0,T), \\ %\frac{1}{p} , &&\text{ on } \Omegaair \times (0,T),\\
		%\div_\H  \bar{V}^{\ocn} &=F_4(V^\ocn), &&\text{ in } G \times (0,T),\\
     %  \dz  V^\ocn &= 0 &&\text{ on } \Gab \times (0,T), \\
     %   \dz V^\ocn &= B_2(\zeta, V^\air,V^\ocn), &&\text{ on } \Gau \times (0,T), \\
%        , \\
%        (\dz \vocn, \wocn) &=(0, 0), &&\text{ on } \Gamma_b \times (0,T), \\
     \xi(0)=\xi_0, \ V(0) = V_0, \ T(0) &= T_0,
	\end{aligned}
	\right. 
	\label{eq: full CPE Lagrange linear}
\end{equation}
with prescribed forcing terms $g_1, g_2$ and $g_3$ and boundary conditions \begin{equation}\label{eq: bc linear}
\partial_z V|_{z=0,1}=0 \ \text{ and }\ \partial_z T|_{z=0,1}=0.
\end{equation}
% Define the ground space $\rX_0$ by 
% \begin{equation*}
%     \rX_0:= \rL^p(0,T;\rLq(\Omega)) \times \rL^p(0,T;\rLq(\Omega;\R^2)) \times \rL^p(0,T;\rLq(\T^2))
% \end{equation*}
% and the domain $\rX_1$ by 
% \begin{equation*}
%     \rX_1 := \rH_\rN^{2,q}(\Omega) \times \rH_\rN^{2,q}(\Omega;\R^2) \times \rH^{2,q}(\T^2),
% \end{equation*}
% where the index $_\rN$ represents the Neumann boundary conditions, i.\ e., $\varrho \in \rH^{2,q}$ and $\dz \varrho|_{z=0,1}=0$. The \emph{data} and \emph{solution} spaces are then given by 
% \begin{equation*}
%     \E_0(0,T) := \rLp(0,T;\rX_0) \ \text{ and } \ \E_1(0,T) := \rLp(0,T;\rX_1) \cap \rH^{1,p}(0,T;\rX_0).
% \end{equation*}
% Next, we define the initial data space $\rX_\beta$, which is given by real interpolation of the ground space $\rX_0$ and the domain $\rX_1$, that is, 
% \begin{equation*}
%     \rX_\beta := ( \rX_0, \rX_1)_{1- \nicefrac{1}{p},p},
% \end{equation*}
% $(\cdot, \cdot)_{\theta,p}$ denotes the real interpolation functor for $\theta \in (0,1)$ and $p\in (1,\infty)$.
The next part of this section is devoted to proving a well-posedness result concerning the linearized system \eqref{eq: full CPE Lagrange linear} subject to the boundary condition \eqref{eq: bc linear}. Before stating the corresponding result we define the ground space $\rX_0 :=\rH^3(\T^2) \times \rH^2(\Omega)^2 \times \rH^2(\Omega) $ as well as the domain $\rX_1 := \rH^3(\T^2) \times \rH^4_\rN(\Omega)^2 \times \rH^4_\rN(\Omega) $, where the subscript $_\rN$ denotes Neumann boundary conditions, that is, $\rH^s_\rN (\Omega)  := \{ \rH^s(\Omega) \colon \dz f|_{z=0,1}=0 \}$ for suitable $s>0$ such that the trace is well defined. We then define the full data and solution spaces by 
\begin{equation*}
    \E_0(0,\tau):= \rL^2(0,\tau;\rX_0) \ \text{ and } \ \E_1(0,\tau) := \rL^2(0,\tau;\rX_1) \cap \rH^1(0,\tau;\rX_0).
\end{equation*}
We are now in a position to state the main result of this section.
\begin{lem}
    \label{lem: linear max reg}
    Let $\tau>0$ and $\varrhobar^\ast, \Theta^\ast >0$. Assume that 
    \begin{equation*}
        (\xi_0,V_0,T_0) \in \rH^3(\T^2) \times \rH^3(\Omega)^2 \times \rH^3(\Omega) \ \text{ and } \ (g_1,g_2,g_3) \in \E_0(0,\tau)
    \end{equation*}
    where the initial data is satisfying the compatibility conditions
    \begin{equation*}
        \partial_z V_0|_{z=0,1}=0 \ \text{ and }\ \partial_z T_0|_{z=0,1}=0.
    \end{equation*}
    Then there is a unique, strong solution $(\xi,V,T) \in \E_1(0,\tau) $ satisfying
    \begin{equation*}
        \|  (\xi,V,T)  \|_{\E_1(0,\tau)} \leq C \big ( \| (g_1,g_2,g_3)\|_{  \E_0(0,\tau)} + \| (\xi_0, V_0,T_0) \|_{ \rH^3(\T^2) \times \rH^3(\Omega)^2 \times \rH^3(\Omega) } \big )
    \end{equation*}
    for a constant $C=C(\tau, \varrhobar^\ast, \Theta^\ast,\Omega)>0$.
\end{lem}
\begin{proof}
  The main technical difficulty of the argument lies in the temperature equation, since it contains the nonlocal operator $\mathcal{L}$ acting on the time derivative $\dt \Theta$. For this reason, we divide the proof into two steps: first we analyze the temperature equation in isolation, and then we turn to the full coupled system.
   \begin{step}{\emph{The Temperature Equation}} \mbox{}\\
 We introduce the operator $\mathcal{P}$ on $\rH^{2}(\Omega)$ by
\begin{equation*}
    (\mathcal{P}f)(\cdot,z) := \beta(z)\int_{0}^{1} f(\cdot,\eta)\,\d\eta,
    \qquad f\in \rH^{2}(\Omega),
\end{equation*}
where $\beta\in \rC^\infty([0,1])$ satisfies $\int_{0}^{1}\beta(z)\,\d z = 1$. A direct computation shows that $\mathcal{P}$ is idempotent, namely,
\begin{equation*}
(\mathcal{P}^2 f)(\cdot,z)
= \beta(z)\int_0^1 (\mathcal{P}f)(\cdot,\eta)\,\d\eta
= \beta(z)\int_0^1 \beta(\eta)\Big(\int_0^1 f(\cdot,\zeta)\,\d\zeta\Big)\,\d\eta
= (\mathcal{P}f)(\cdot,z),
\end{equation*}
hence $\mathcal{P}$ is a projection on $\rH^2(\Omega)$. Moreover, $\mathcal{P}$ is bounded, since $\beta$ is bounded and vertical integration is a bounded operation on $\rH^2(\Omega)$. With this notation we can write
\begin{equation*}
    \mathcal{L} = 2\,\mathrm{Id}-\mathcal{P},
\end{equation*}
so that $\mathcal{L}$ is invertible with
\begin{equation*}
    \mathcal{L}^{-1}=\frac12\big(\mathrm{Id}+\mathcal{P}\big).
\end{equation*}
In particular, $\mathcal{L}^{-1}$ has the two-point spectrum $\{1/2,1\}$. Therefore, for every bounded holomorphic function $f$ defined on a domain containing $\sigma(\mathcal{L}^{-1})$, the holomorphic functional calculus yields
\begin{equation*}
    f(\mathcal{L}^{-1}) = f(1)\,\mathcal{P} + f(1/2)\,(\mathrm{Id}-\mathcal{P}).
\end{equation*}
Consequently, $\mathcal{L}^{-1}$ admits a bounded $\Hinfty$-calculus on $\rH^2(\Omega)$ with $\Phi^\infty_{\mathcal{L}^{-1}}=0$, and hence it is $\mathcal{R}$-sectorial with $\Phi^\mathcal{R}_{\mathcal{L}^{-1}}=0$.

With these preparations, we apply $\mathcal{L}^{-1}$ to the temperature equation after modifying by a positive shift. This gives
\begin{equation}\label{eq: Temp with L}
    \dt T + \mathcal{L}^{-1} \big (-\alpha \,\Delta  + \omega \big ) T -\omega \mathcal{L}^{-1} T=\mathcal{L}^{-1}\big ( g_3 + \varrhobar^\ast\alpha \, \mathcal{I}_z ( \divH V)-\Theta^\ast \mathrm{exp}(z/\Theta^\ast)  \;
\mathcal{I}_1 (\divH V)\big ).
\end{equation}
Our goal is to show that $\mathcal{L}^{-1}(-\alpha\,\Delta+\omega)$ is $\mathcal{R}$-sectorial on $\rH^2(\Omega)$ with angle strictly less than $\pi/2$. We treat it as a product of the (in general non-commuting) operators $\mathcal{L}^{-1}$ and $-\alpha\,\Delta+\omega$.
We first recall that
\begin{equation*}
   -\alpha\, \Delta + \omega \colon \rH^2(\Omega) \to \rH^2(\Omega),
   \qquad \D(-\alpha\, \Delta + \omega)= \rH^4_\rN (\Omega),
\end{equation*}
is $\mathcal{R}$-sectorial on $\rH^2(\Omega)$, since $\alpha$ is smooth and bounded away from $0$. More precisely, $\mathcal{R}$-sectoriality is known on the half-space and extends to the layer $\Omega_L=\R^2\times(0,1)$ by localization, cf.\ \cite{DHP:03}. Passing to $\Omega=\T^2\times(0,1)$ is then achieved by extending periodic functions to $\R^2$, as explained on p.~1082 in \cite{HK:16}.
Since $\mathcal{L}^{-1}$ is bounded on $\rH^2(\Omega)$, we define
\begin{equation*}
    S := AB \colon \rH^2(\Omega) \to \rH^2(\Omega),
    \qquad \D(S):=\rH^4_\rN(\Omega),
    \qquad A=\mathcal{L}^{-1},\ \ B=-\alpha\,\Delta+\omega.
\end{equation*}
To verify $\mathcal{R}$-sectoriality for $S$ (up to a shift), we consider the commutator expression
\begin{equation*}
    Z_{A,B}(\lambda,\mu)
    :=\big[A(\mu+B)^{-1}-(\mu+B)^{-1}A\big](\lambda+A)^{-1},
\end{equation*}
and show the estimate
\begin{equation}\label{eq: commutator}
      \| Z_{A,B} (\lambda,\mu) \|
      \leq \frac{C}{(1+|\lambda|)^{1-a} |\mu|^{1+b}},
      \ \text{ for all }\ (\lambda,\mu) \in \Sigma_{\pi} \times \Sigma_{\pi-\Phi^\mathcal{R}_B},
\end{equation}
with $C,a,b\ge 0$ and $a+b<1$. This allows us to apply \cite[Theorem~1.1]{Weber:98} and obtain sectoriality of $S$ up to a shift. Since $\rH^2(\Omega)$ is a Hilbert space, sectoriality and $\mathcal{R}$-sectoriality coincide, cf.\ \cite[Chapter~4]{PS:16}.
Indeed, since $A$ is boundedly invertible with spectrum $\{1/2,1\}$ and $\mathcal{R}$-sectorial on $\rH^2(\Omega)$ of angle $0$, there exists $C_A>0$ such that
\begin{equation*}
      \| (\lambda + A)^{-1} \| \leq \frac{C_A}{1+ | \lambda|},
      \ \text{ for all } \ \lambda \in \Sigma_\pi.
\end{equation*}
Likewise, $\mathcal{R}$-sectoriality of $B$ with $\Phi^\mathcal{R}_B<\pi/2$ yields $C_B>0$ such that
\begin{equation*}
        \| (\mu + B)^{-1} \| \leq \frac{C_B}{ | \mu|},
        \ \text{ for all } \ \mu \in \Sigma_{\pi-\Phi^\mathcal{R}_B}.
\end{equation*}
Writing $R_\mu=(\mu+B)^{-1}$, we obtain the crude commutator bound
\begin{equation*}
      \|AR_\mu-R_\mu A\| \le 2\|A\|\,\|R_\mu\|,
\end{equation*}
so that \eqref{eq: commutator} holds with $a=b=0$. Finally, the term $\omega A$ is of lower order and can be treated as a perturbation. Hence, by \cite[Proposition~4.3]{DHP:03}, there exists a shift $\omega_1\ge 0$ such that $S-\omega A+\omega_1$ is $\mathcal{R}$-sectorial on $\rH^2(\Omega)$ with angle strictly less than $\pi/2$.

   \end{step}
   \begin{step}{\emph{The full Operator Matrix}}\mbox{}\\
   In this step, we consider the full abstract Cauchy problem corresponding to the linear system \eqref{eq: full CPE Lagrange linear} which is given by 
   \begin{equation*}
       \dt \mathbf{u} + A \mathbf{u}= \mathbf{g},
   \end{equation*}
   where $\mathbf{u}= (\xi, V, T)^\top$, $\mathbf{g}= (g_1,g_2,g_3)^\top$ and the operator matrix $A$ is defined by 
   \begin{equation*}
       A \colon \rX_0 \to \rX_0, \quad \D(A) := \rX_1
   \end{equation*}
   and 
   \begin{equation*}
       A = \begin{pmatrix}
           0 & - \varrhobar^\ast \, \mathcal{I}_1 ( \divH) & 0 \\
            \Theta^\ast\frac{\nablaH }{\varrhobar^\ast} & - \mu \, \alpha\, \Delta - \mu' \, \alpha\, \nablaH \divH & \mathcal{A} \\ 
           0 &- \mathcal{L}^{-1}\big ( \varrhobar^\ast \alpha \,\mathcal{I}_z (\divH ) - \Theta^\ast \mathrm{exp}(z/\Theta^\ast) \mathcal{I}_1 (\divH)\big ) & - \mathcal{L}^{-1}( \alpha \, \Delta -\omega ) - \omega \mathcal{L}^{-1} 
       \end{pmatrix}
   \end{equation*}
To show that up to a shift the operator matrix $A$ is $\mathcal{R}$-sectorial on $\rX_0$ of angle strictly less than $\pi/2$, we split the matrix into a main part $A_0$ and a perturbative part $B$, that is, $A=A_0 +B$, where the respective parts are given by 
\begin{equation*}
A_0:=\begin{pmatrix}
      0 & - \varrhobar^\ast \, \mathcal{I}_1 ( \divH) & 0 \\
          0 & - \mu \, \alpha\, \Delta - \mu' \, \alpha\, \nablaH \divH & 0 \\ 
           0 &0 & - \mathcal{L}^{-1}( \alpha \, \Delta -\omega ) - \omega \mathcal{L}^{-1} 
\end{pmatrix} \ \text{ and } \ B:= \begin{pmatrix}
      0 &0 & 0 \\
           \Theta^\ast\frac{\nablaH }{\varrhobar^\ast} &0 & \mathcal{A} \\ 
           0 & \mathcal{C} &0
\end{pmatrix}
\end{equation*}
with $\mathcal{C}:= - \mathcal{L}^{-1}\big ( \varrhobar^\ast \alpha \,\mathcal{I}_z (\divH ) - \Theta^\ast \mathrm{exp}(z/\Theta^\ast) \mathcal{I}_1 (\divH)\big )$. Note that using $\sup_{z \in [0,1]}| \hat{B}(\Theta^\ast)(z) | \leq M$ and Jensen's inequality we obtain 
\begin{equation*}
    \| \varrhobar^\ast \mathcal{I}_1 (\divH V) \|_{\rH^3(\T^2)} \leq C \big \| \int_0^1 [ \hat{B}(\Theta^\ast) \divH V \big ] (\cdot, \eta ) \d \eta \big \|_{\rH^3(\T^2)} \leq C \int_0^1 \| V(\cdot,\eta) \|_{\rH^4(\T^2)} \d \eta \leq \| V \|_{\rH^4(\Omega)}
\end{equation*}
and therefore $ - \varrhobar^\ast \, \mathcal{I}_1 ( \divH)$ is a bounded operator from $\rH^4(\Omega)$ to $\rH^3(\T^2)$. Next, we note that the differential operator $-\mu\,\alpha\,\Delta-\mu'\,\alpha\,\nablaH\divH$ coincides with the \emph{hydrostatic Lam\'e operator} introduced in \cite{HIRZ:25}. It is shown there that, after a suitable shift, this operator admits a bounded $\Hinfty$-calculus on $\rL^q(\Omega)$ with $\Hinfty$-angle strictly smaller than $\pi/2$. The same strategy can be adapted to the present setting and yields the corresponding statement on the ground space $\rH^2(\Omega)$. In particular, there exists a shift $\omega_2\ge 0$ such that $-\mu\,\alpha\,\Delta-\mu'\,\alpha\,\nablaH\divH+\omega_2$ is $\mathcal{R}$-sectorial on $\rH^2(\Omega)$ with angle strictly less than $\pi/2$.
Combining this with the conclusion of Step~1, we set $\omega_0\coloneqq \max\{\omega_1,\omega_2\}$ and obtain that the operator matrix $A_0+\omega_0$ is $\mathcal{R}$-sectorial on $\rX_0$ with angle strictly less than $\pi/2$.

To obtain the corresponding statement for the full operator matrix $A$, we estimate the perturbation $B$ in the natural ground space. More precisely, for every $\mathbf{u}\in\D(A)$ we have
\begin{equation*}
   \| B \mathbf{u} \|_{\rX_0} \leq C \big ( \| \nablaH \xi \|_{\rH^2(\Omega)} + \| \mathcal{A}T \|_{\rH^2(\Omega)} + \| \mathcal{C}V \|_{\rH^2(\Omega)} \big ).
\end{equation*}
Since $\xi$ is independent of the vertical variable, it follows immediately that $\| \nablaH \xi \|_{\rH^2(\Omega)} \leq C \| \xi \|_{\rH^3(\T^2)}$. 
Moreover, using the definition of $\mathcal{A}$ and the boundedness of the operator $(D\hat{B})(\Theta^\ast)$, we obtain
\begin{equation*}
    \| \mathcal{A}T \|_{\rH^2(\Omega)} = \big \|  \Big ( \Theta^\ast\frac{ (D\hat{B})(\Theta^\ast)}{\hat{B}(\Theta^\ast)} + \mathrm{I}_2\Big ) \nablaH T \big\|_{\rH^2(\Omega)} \leq C \| \nablaH T \|_{\rH^2(\Omega)} \leq C\| T \|_{\rH^3(\Omega)}. 
\end{equation*}
Finally, invoking the boundedness of $\mathcal{L}^{-1}$ from Step~1, the uniform boundedness of the coefficients $\alpha$ and $\beta$, and Jensen's inequality, we infer
\begin{equation*}
    \begin{aligned}
        \| \mathcal{C} V\|_{\rH^2(\Omega)} &= \|  - \mathcal{L}^{-1}\big ( \varrhobar^\ast \alpha \,\mathcal{I}_z (\divH ) - \Theta^\ast \mathrm{exp}(z/\Theta^\ast) \mathcal{I}_1 (\divH)\big ) V \|_{\rH^2(\Omega)} \\ 
        &\leq C\Big ( \big \| \int_0^z  [ \hat{B}(\Theta^\ast) \divH V \big ] (\cdot, \eta ) \d \eta \big \|_{\rH^2(\Omega)} + \big \| \int_0^1 [ \hat{B}(\Theta^\ast) \divH V \big ] (\cdot, \eta ) \d \eta \big \|_{\rH^2(\Omega)} \Big ) \\
        &\leq C \Big ( \int_0^z \d \eta \;\| \divH V\|_{\rH^2(\Omega)} + \int_0^1 \d \eta \; \| \divH V \|_{\rH^2(\Omega)} \Big )
        \\ 
        &\leq C \| V \|_{\rH^3(\Omega)}.
    \end{aligned}
\end{equation*}
Combining these bounds yields
\begin{equation*}
 \begin{aligned}
         \| B \mathbf{u} \|_{\rH^3(\T^2)\times \rH^2(\Omega)\times \rH^2(\Omega)} 
         &\leq C \big ( \| \xi \|_{\rH^3(\T^2)} + \| T \|_{\rH^3(\Omega)} + \| V \|_{\rH^3(\Omega)} \big ).
 \end{aligned}
\end{equation*}
Using interpolation and Young's inequality, we obtain for every $\eps>0$ the estimate
\begin{equation*}
 \begin{aligned}
         \| B \mathbf{u} \|_{\rX_0} 
         &\leq \eps \big ( \| T \|_{\rH^4(\Omega)} + \| V \|_{\rH^4(\Omega)} \big ) + C(\eps) \big ( \| \xi \|_{\rH^3(\T^2)} + \| T \|_{\rH^2(\Omega)} + \| V \|_{\rH^2(\Omega)} \big ) \\ 
         & \leq \eps \| A_0  \mathbf{u}\|_{\rX_0} + C(\eps) \| \mathbf{u} \|_{\rX_0}.
 \end{aligned}
\end{equation*}
Hence $A$ can be viewed as a relatively $A_0$-bounded perturbation. Standard perturbation theory (see, for example, \cite[Proposition~4.3]{DHP:03}) therefore yields the existence of $\omega \geq 0$ such that $A + \omega$ is $\mathcal{R}$-sectorial on ${\rX_0}$ with $\Phi^\mathcal{R}_{A+\omega} < \pi /2$. The result now follows since $\mathcal{R}$-sectoriality is equivalent to maximal regularity. In particular, the initial data stems from the real interpolation space $(\rX_0,\rX_1)_{1/2, 2}$ where $(\cdot, \cdot)_{\theta,p}$ denotes the interpolation functor for $\theta \in (0,1)$ and $p \in (1,\infty)$. Specifically, $(\rX_0,\rX_1)_{1/2, 2}= \rH^3(\T^2) \times \rH^3_\rN(\Omega)^2 \times \rH^3_\rN(\Omega)$. 

   \end{step}

\end{proof}
\noindent
The last remaining step toward global existence, under the assumption that the solution remains close to the constant equilibrium determined by $\varrhobar^\ast$ and $\Theta^\ast$, is to control the nonlinear remainder terms $(f_1,f_2,f_3)$ generated by the Lagrangian transformation introduced in \autoref{sec: Lagrange}. This is carried out in the following lemma.
\begin{lem}\label{lem: est nonlinear}
    Let $\tau >0$ and $\varrhobar^\ast, \Theta^\ast >0$. Assume that 
    \begin{equation*}
      (  \varrhobar^\rL, v^\rL ,\Theta^\rL ) \in \mathbb{B}_\eps(0,\tau) := \{  (  \varrhobar^\rL, v^\rL ,\Theta^\rL )  \in \E_1(0,\tau) \colon \|  (  \varrhobar^\rL, v^\rL ,\Theta^\rL ) \|_{\E_1(0,\tau)} \leq \eps \}.
    \end{equation*}
    Then there is a constant $C=C(\tau, \varrhobar^\ast, \Theta^\ast, \Omega)>0$ such that 
    \begin{equation*}
        \| (f_1,f_2,f_3) (\varrhobar^\rL, v^\rL, \Theta^\rL) \|_{\E_0(0,\tau)} \leq C\eps^2.
    \end{equation*}
\end{lem}
\begin{proof}
We begin by collecting a few auxiliary estimates that will be used repeatedly in the sequel. First, we note the continuous embedding
\begin{equation*}
    \E_1(0,\tau) \hookrightarrow \rL^\infty \bigl(0,\tau;\, \rH^3(\T^2)\times \rH^3(\Omega)^2 \times \rH^3(\Omega)\bigr),
\end{equation*}
which will be invoked whenever needed.
From the discussion in \autoref{sec: Lagrange} it readily follows that 
\begin{equation*}
    \| \delta \hat{B}(\Theta^\rL, \Theta^\ast) \|_{\E^\Theta_1(0,\tau)} + \| \delta (D \hat{B})(\Theta^\rL,\Theta^\ast) \|_{\mathrm{op}} \leq C \eps 
\end{equation*}
for a constant $C=C(\tau, \Theta^\ast, \Omega)>0$. In particular, we infer from \eqref{eq: divH b} that 
\begin{equation*}
  \begin{aligned}
        \| \divH b^\rL \|_{\rL^2_\tau \rH^3(\T^2)} &\leq C \big ( \| D\hat{B}(\Theta^\ast)(\nablaH \Theta^\rL) \|_{\rL^2_\tau \rH^3(\T^2)} + \| \delta (D \hat{B})(\Theta^\rL,\Theta^\ast) \|_{\mathrm{op}} \big ) \| v^\rL \|_{\E^v_1(0,\tau)} \\ 
        &\quad+ C \big ( \| \hat{B}(\Theta^\ast) \|_{\E^\Theta_1(0,\tau)} + \| \delta \hat{B}(\Theta^\rL, \Theta^\ast) \|_{\E^\Theta_1(0,\tau)} \big ) \| v^\rL \|_{\E^v_1(0,\tau)} \\
        &\leq C \eps + C\eps^2,
  \end{aligned}
\end{equation*}
where we used that the underlying space $\rH^2(\T^2)$ is an algebra with respect to pointwise multiplication. Next, we address the terms involving the vertical velocity contribution $(\varrhobar\,\hat{B}\,w)^{\rL}$. More precisely, we estimate the quantities $J_1$ and $J_2$ introduced in \eqref{eq J_1 and J_2}. For $J_1$ we obtain
\begin{equation*}
    \begin{aligned}
       \| J_1 \|_{\rL^2_\tau \rH^2(\Omega)}
       &\leq  \Big \| \int_0^z \big [(D\hat{B})(\Theta^\ast)\,(\dt \Theta^\rL) + \hat{B}(\Theta^\ast)\,\divH v^\rL  \big ] (\cdot , \eta)\, \d \eta \Big \|_{\rL^2_\tau \rH^2(\Omega)} \\
       &\quad  + \Big \| \Big(\int_0^z \hat{B}(\Theta^\ast)(\cdot,\eta)\,\d\eta\Big) 
       \Big(\int_0^1\big [ \hat{B}(\Theta^\ast)\, \divH v^\rL\big ](t,\cdot,\zeta)\,\d\zeta\Big) \Big \|_{\rL^2_\tau \rH^2(\Omega)} \\
       & \leq C \| (D\hat{B})(\Theta^\ast)\,(\dt \Theta^\rL) + \hat{B}(\Theta^\ast)\,\divH v^\rL \|_{\rL^2_\tau \rH^2(\Omega)}
       \leq C\eps .
    \end{aligned}
\end{equation*}
By the same arguments, using additionally \autoref{lem:ests of trafo}, we infer
\begin{equation*}
    \| J_2 \|_{\rL^2_\tau \rH^2(\Omega)}  \leq C \eps^2 .
\end{equation*}
In particular, $J_1$ and $J_2$ are vertically integrated and therefore their respective vertical derivatives can be estimated the same way, that is, 
\begin{equation*}
    \| \dz J_1 \|_{\rL^2_\tau \rH^2(\Omega)} \leq C\eps \ \text{ and } \ \| \dz J_2 \|_{\rL^2_\tau \rH^2(\Omega)} \leq C\eps^2.
\end{equation*}
From the estimates on $J_1$ and $J_2$ we conclude that
\begin{equation*}
    \| (\varrhobar \, \hat{B} \, w)^{\rL}\|_{\rL^2_\tau \rH^2(\Omega)}
    \leq \| J_1 + J_2 \|_{\rL^2_\tau \rH^2(\Omega)}
    \leq C \eps  \ \text{ and }\  \| \dz(\varrhobar \, \hat{B} \, w)^{\rL}\|_{\rL^2_\tau \rH^2(\Omega)} \leq C\eps.
\end{equation*}
We now turn to the first contribution in the remainder term $f_1$. Estimating it in the same spirit as $\divH b^\rL$, we obtain
\begin{equation*}
    \begin{aligned}
      &\quad \Big \| \int_0^1 \Big[ \Big(\big (D\hat{B} + \delta ( D\hat{B}) \big ) ( \nablaH \Theta^\rL) \Big)\,v^\rL
      +  \delta \hat{B}\,\divH v^\rL \Big](\cdot,\eta)\, \d \eta \Big \|_{\rL^2_\tau \rH^3(\T^2)} \\ 
      &\leq C \Big( \| D\hat{B}(\Theta^\ast)(\nablaH \Theta^\rL) \|_{\rL^2_\tau \rH^3(\T^2)}
      + \| \delta (D \hat{B})(\Theta^\rL,\Theta^\ast) \|_{\mathrm{op}}
      + \| \delta \hat{B}(\Theta^\rL, \Theta^\ast) \|_{\E^\Theta_1(0,\tau)} \Big)\, \| v^\rL \|_{\E^v_1(0,\tau)} \\
      &\leq C \eps^2 .
    \end{aligned}
\end{equation*}
Finally, invoking \autoref{lem:ests of trafo}, we estimate the remaining contribution by
\begin{equation*}
    \begin{aligned}
          &\quad\big \|  \varrhobar^\rL \divH b^\rL
          + (\varrhobar^\rL + \varrhobar^\ast)\, \nablaH b^\rL : (\mathrm{Z}^\top  - \mathrm{I}_2 )\big \|_{\rL^2_\tau \rH^3(\T^2)} \\ 
          &\leq C \| \varrhobar^\rL \|_{\rL^\infty_\tau \rH^3(\T^2)} \| \divH b^\rL \|_{\rL^2_\tau \rH^3(\T^2)}
          + C \| \varrhobar^\rL + \varrhobar^\ast \|_{\rL^\infty_\tau \rH^3(\T^2)} \| \nablaH b^\rL \|_{\rL^2_\tau \rH^3(\T^2)} \| \mathrm{Z}^\top  - \mathrm{I}_2 \|_{\rL^\infty_\tau \rH^3(\T^2)} \\ 
          & \leq C\eps^2 ,
    \end{aligned}
\end{equation*}
which completes the desired estimate of $f_1$.
We next turn to $f_2$ and begin with the term
\begin{equation*}
    \begin{aligned}
        &\quad\Big \|  \Big (\frac{\varrhobar^\rL \,\delta\hat{B}}{\varrhobar^\ast \,\hat{B}(\Theta^\ast)}
        + \frac{\varrhobar^\rL}{\varrhobar^\ast}
        + \frac{\delta\hat{B}} {\hat{B}(\Theta^\ast)}\Big )\, \dt v^\rL \Big \|_{\rL^2_\tau \rH^2(\Omega)} \\
        & \leq C \Big (\| \varrhobar^\rL \|_{\rL^\infty_\tau \rH^2(\T^2)} \| \delta \hat{B} \|_{\rL^\infty_\tau \rH^2(\Omega)}
        + \| \varrhobar^\rL \|_{\rL^\infty_\tau \rH^2(\T^2)}
        + \| \delta \hat{B} \|_{\rL^\infty_\tau \rH^2(\Omega)} \Big ) \| \dt v^\rL \|_{\rL^2_\tau \rH^2(\Omega)} \\
        & \leq C\eps^2 .
     \end{aligned}
\end{equation*}
Moreover,
\begin{equation*}
    \begin{aligned}
        &\quad \Big \| \frac{(\varrhobar^\rL + \varrhobar^\ast)\, ( \hat{B}(\Theta^\ast) +\delta \hat{B}) \, \big ( \mathrm{Z}(v^\rL - b^\rL) \cdot \nablaH \big ) v^\rL}{\varrhobar^\ast \,\hat{B}(\Theta^\ast)} \Big \|_{\rL^2_\tau \rH^2(\Omega)} \\
        &\leq C \| \varrhobar^\rL + \varrhobar^\ast \|_{\rL^\infty_\tau \rH^2(\Omega)}
        \| \hat{B}(\Theta^\ast) + \delta \hat{B} \|_{\rL^\infty_\tau \rH^2(\Omega)}
        \| \mathrm{Z} \|_{\rL^\infty_\tau \rH^2(\Omega)}
        \big ( \| v^\rL \|_{\rL^\infty_\tau \rH^2(\Omega)} + \| b^\rL \|_{\rL^\infty_\tau \rH^2(\Omega)} \big )
        \| \nablaH v^\rL \|_{\rL^2_\tau \rH^2(\Omega)} \\
        &\leq C\eps^2 .
    \end{aligned}
\end{equation*}
The transformed viscous contributions are estimated by
\begin{equation*}
    \begin{aligned}
        &\quad \Big \| \frac{\mu}{\varrhobar^\ast \,\hat{B}(\Theta^\ast)}\Bigl ( \sum_{j,k,l} \frac{\partial^2 v_i^\rL}{\partial y_k \partial y_l} \bigl (\mathrm{Z}_{k,j} - \delta_{k,j}\bigr ) \mathrm{Z}_{l,j}
        + \sum_{k,l}\frac{\partial^2 v_i^\rL}{\partial y_k \partial y_l}\bigl (\mathrm{Z}_{l,k}- \delta_{l,k}\bigr )
        +  \sum_{j,k,l} \mathrm{Z}_{l,j} \frac{\partial v_i^\rL}{\partial y_k} \frac{\partial \mathrm{Z}_{k,j}}{\partial y_l} \Bigr ) \Big \|_{\rL^2_\tau \rH^2(\Omega)} \\
        &\leq C \| v^\rL \|_{\rL^2_\tau \rH^4(\Omega)}\Big (  \| \mathrm{Z}- \mathrm{I}_2 \|_{\rL^\infty_\tau \rH^2(\T^2)} \| \mathrm{Z} \|_{\rL^\infty_\tau \rH^2(\T^2)}
        + \| \mathrm{Z}- \mathrm{I}_2 \|_{\rL^\infty_\tau \rH^2(\T^2)}
        + \| \mathrm{Z} \|_{\rL^\infty_\tau \rH^2(\T^2)} \Bigl\| \frac{\partial \mathrm{Z}_{l,j}}{\partial y_k} \Bigr\|_{\rL^\infty_\tau\rH^2(\T^2)} \Big ) \\
        &\leq C\eps^2 .
 \end{aligned}
\end{equation*}
In exactly the same way,
\begin{equation*}
    \begin{aligned}
        &\quad \Big \| \frac{\mu'}{\varrhobar^\ast \,\hat{B}(\Theta^\ast)}
        \Bigl (\sum_{l,k,j} \frac{\partial^2 v^\rL_j}{\partial y_k \partial y_l}\bigl ( \mathrm{Z}_{l,i}- \delta_{l,i}\bigr )\mathrm{Z}_{k,j}
        +  \sum_{k,l}\frac{\partial^2 v^\rL_l}{\partial y_k \partial y_l} \bigl ( \mathrm{Z}_{l,i}- \delta_{l,i}\bigr )
        +  \sum_{j,k,l} \mathrm{Z}_{l,i} \frac{\partial v^\rL_j}{\partial y_k} \frac{\partial \mathrm{Z}_{k,j}}{\partial y_l} \Bigr )\Big \|_{\rL^2_\tau \rH^2(\Omega)} \\
        &\leq C \| v^\rL \|_{\rL^2_\tau \rH^4(\Omega)}\Big (  \| \mathrm{Z}- \mathrm{I}_2 \|_{\rL^\infty_\tau \rH^2(\T^2)} \| \mathrm{Z} \|_{\rL^\infty_\tau \rH^2(\T^2)}
        + \| \mathrm{Z}- \mathrm{I}_2 \|_{\rL^\infty_\tau \rH^2(\T^2)}
        + \| \mathrm{Z} \|_{\rL^\infty_\tau \rH^2(\T^2)} \Bigl\| \frac{\partial \mathrm{Z}_{l,j}}{\partial y_k} \Bigr\|_{\rL^\infty_\tau\rH^2(\T^2)} \Big ) \\
        &\leq C\eps^2 .
    \end{aligned}
\end{equation*}
Finally, for the contribution involving the vertical velocity we have
\begin{equation*}
    \Big \|  \frac{(\varrhobar \, \hat{B} \, w)^{\rL} \, (\dz v^\rL)_i}{\varrhobar^\ast \,\hat{B}(\Theta^\ast)} \Big \|_{\rL^2_\tau \rH^2(\Omega)}
    \leq C \| (\varrhobar \, \hat{B} \, w)^{\rL} \|_{\rL^2_\tau \rH^2(\Omega)} \| \dz v^\rL \|_{\rL^\infty_\tau \rH^2(\Omega)}
    \leq C\eps^2 .
\end{equation*}
All remaining contributions stemming from the pressure gradient are of higher order and can be treated in the same manner.

We now turn to the temperature remainder $f_3$. We begin with the term arising from the vertical pressure work and estimate
\begin{equation*}
   \begin{aligned}
        &\quad \big \| \frac{ \dz \Theta^\rL J_1+ (\dz \Theta^\rL + \Theta^\ast + 1) J_2 +\Theta^\ast \dz J_2 + \Theta^\rL \dz J_1}{\varrhobar^\ast \,\hat{B}(\Theta^\ast)} \big \|_{\rL^2_\tau \rH^2(\Omega)} \\
        &\leq C \Big( \| \dz \Theta^\rL \|_{\rL^\infty_\tau \rH^2(\Omega)} \| J_1 \|_{\rL^2_\tau \rH^2(\Omega)}
        + \| \dz \Theta^\rL + \Theta^\ast + 1 \|_{\rL^\infty_\tau \rH^2(\Omega)} \| J_2 \|_{\rL^2_\tau \rH^2(\Omega)}
        + \| \dz J_2 \|_{\rL^2_\tau \rH^2(\Omega)}  \\
        &\qquad\qquad + \| \Theta^\rL \|_{\rL^\infty_\tau \rH^2(\Omega)} \| \dz J_1 \|_{\rL^2_\tau \rH^2(\Omega)} \Big)
        \leq C \eps^2 .
   \end{aligned}
\end{equation*}
The remaining terms in $f_3$ have the same structure as the corresponding contributions in $f_2$ and can therefore be estimated analogously. This completes the bounds for the nonlinear remainder terms.
\end{proof}
\noindent
Arguing along the lines of \autoref{lem: est nonlinear}, we can establish the following stability estimates for differences of solutions. Since the proof follows the same strategy, we omit the details.
\begin{cor}\label{cor: lipschitz nonlinear}
     Let $\tau >0$ and $\varrhobar^\ast, \Theta^\ast >0$. Assume that 
    \begin{equation*}
      (  \varrhobar^\rL_i, v^\rL_i ,\Theta^\rL_i ) \in \mathbb{B}_\eps(0,\tau) \ \text{ for } \ i =1,2.
    \end{equation*}
    Then there is a constant $C=C(\tau, \varrhobar^\ast, \Theta^\ast, \Omega)>0$ such that 
    \begin{equation*}
        \| (f_1,f_2,f_3) (\varrhobar^\rL_1, v^\rL_1, \Theta^\rL_1) -(f_1,f_2,f_3) (\varrhobar^\rL_2, v^\rL_2, \Theta^\rL_2)  \|_{\E_0(0,\tau)} \leq C\eps \|   (\varrhobar^\rL_1, v^\rL_1, \Theta^\rL_1) - (\varrhobar^\rL_2, v^\rL_2, \Theta^\rL_2)  \|_{\E_1(0,\tau)}.
    \end{equation*}
\end{cor}
\noindent
\section{Proof of the Main Theorem}\label{sec: proof}
\noindent
Finally, in this section, we assemble all the ingredients and prove the main theorem by means of the contraction mapping principle.
\begin{proof}[Proof of \autoref{thm main}]
Let $\mathbb{B}_\eps(0,\tau)$ be the solution ball introduced in \autoref{lem: est nonlinear}. We define the solution operator $\Psi$ by
\begin{equation*}
    \Psi \colon \mathbb{B}_\eps(0,\tau) \to \E_1(0,\tau), \quad (\varrhobar^\rL_1, v^\rL_1,\Theta^\rL_1) \mapsto \Psi (\varrhobar^\rL_1, v^\rL_1,\Theta^\rL_1) := (\varrhobar^\rL, v^\rL, \Theta^\rL),
\end{equation*}
where $(\varrhobar^\rL, v^\rL, \Theta^\rL) \in \E_1(0,\tau)$ denotes the unique strong solution to the linearized problem \eqref{eq: full CPE Lagrange linear} with right-hand sides $(f_1,f_2,f_3)(\varrhobar^\rL_1, v^\rL_1,\Theta^\rL_1)$. 
By \autoref{lem: linear max reg} and the nonlinear bounds from \autoref{lem: est nonlinear}, the map $\Psi$ is well defined.

We next verify that $\Psi$ maps $\mathbb{B}_\eps(0,\tau)$ into itself and is a contraction, provided the initial data are sufficiently small. More precisely, we assume that
\begin{equation*}
    \| (\varrhobar^\rL_0 , v^\rL_0, \Theta^\rL_0) \|_{\rH^3(\T^2) \times \rH^3(\Omega)^2 \times \rH^3(\Omega)} \leq \frac{\eps}{2C_{\mathrm{lin}}},
\end{equation*}
where $C_{\mathrm{lin}}>0$ denotes the continuity constant from \autoref{lem: linear max reg}. Invoking \autoref{lem: est nonlinear}, we obtain
\begin{equation*}
  \begin{aligned}
        \| (\varrhobar^\rL, v^\rL, \Theta^\rL) \|_{\E_1(0,\tau)} 
        &\leq C_{\mathrm{lin}} \Big( \| (f_1,f_2,f_3)(\varrhobar^\rL_1, v^\rL_1,\Theta^\rL_1) \|_{\E_0(0,\tau)} 
        + \| (\varrhobar^\rL_0 , v^\rL_0, \Theta^\rL_0) \|_{\rH^3(\T^2) \times \rH^3(\Omega)^2 \times \rH^3(\Omega)} \Big) \\
        &\leq C \, C_{\mathrm{lin}} \eps^2 + \frac{\eps}{2}.
  \end{aligned}
\end{equation*}
Hence, choosing $\eps>0$ sufficiently small ensures that $\Psi$ is a self-map, i.e.\ $\Psi \colon \mathbb{B}_\eps(0,\tau) \to \mathbb{B}_\eps(0,\tau)$. The contraction property follows from \autoref{cor: lipschitz nonlinear}. 

Finally, since $\eps>0$ is chosen small, \autoref{lem:ests of trafo} implies that the transform $\rX$ is invertible with inverse $\mathrm{Y}$. Defining the pulled-back functions $(\varrhobar, v ,\Theta)$ by 
\begin{equation*}
    \varrhobar(t,\cdot) := \varrhobar^\rL (t,\mathrm{Y}(\cdot,t)), \ v(t,\cdot,z) := v^\rL (t,\mathrm{Y}(t,\cdot),z) \ \text{ and }\ \Theta(t,\cdot,z):= \Theta^\rL(t,\mathrm{Y}(t,\cdot),z),
\end{equation*}
and introducing the diagnostic quantities $w$, $p$ and $\varrho$ in terms of $(\varrhobar,v,\Theta)$ as described in \autoref{sec : prelim}: We then readily verify that $(\varrho, u=(v,w), \Theta)$ is the unique strong solution of the full compressible primitive equations \eqref{eq:full prim}.

\end{proof}

\end{document}